\documentclass[11pt,a4paper]{article} 

\usepackage{mathpazo}

\usepackage[active]{srcltx}
\usepackage[utf8]{inputenc}
\usepackage[english]{babel}
\usepackage{graphicx}
\usepackage[T1]{fontenc}
\usepackage{amsmath}
\usepackage{amssymb} 
\usepackage{amsthm}
\usepackage{eqnarray}
\usepackage[numbers]{natbib}
\usepackage{array}
\usepackage{esint}
\usepackage{enumitem}
\usepackage{makecell}
\usepackage[left=2.05cm,right=2.05cm,top=2.95cm,bottom=2.8cm]{geometry}
 \usepackage{amssymb}
 \usepackage{bbm, bm}
 \usepackage{amsmath, amsfonts, amsthm, amscd}  
 \usepackage{xcolor}
 \usepackage{mathtools}
 \DeclarePairedDelimiter{\ceil}{\lceil}{\rceil}

\newcommand{\R}{\mathbb{R}}

\newcommand{\N}{\mathbb{N}}

\newcommand{\E}{\mathbb{E}}

\newcommand{\sH}{\mathcal{H}}

\newcommand{\w}{\textbf{w}}
\newcommand{\vs}{\textbf{v}}
\newcommand{\btheta}{\bm{\theta}}

\theoremstyle{plain}

\newtheorem{theorem}{Theorem}

\newtheorem{proposition}{Proposition}
\newtheorem{remark}{Remark}

\newtheorem{lemma}{Lemma}
\newtheorem{corollary}{Corollary}

\usepackage{titlesec}
\titleformat{\subsubsection}[runin]
{\normalfont\itshape}{\thesubsubsection}{1em}{}

 \def\N {\mathbb{N}}

\def \G {\mathbb{G}}
\def \E {\mathbb{E}}
\def \P {\mathbb{P}}
\def\simiid{\overset{\text{i.i.d.}}{\sim}}
\def\simind{\overset{\text{ind.}}{\sim}}

\def \R {\mathbb{R}}
\def \d {\text{d}}

\def \matK {\mathcal{K}}
\def \tP {\tilde{P}}
\def \tF {\tilde{F}}
\def \tQ {\tilde{Q}}
\def \tp {\tilde{p}}
\def \tw {\tilde{w}}
\def \ttheta {\tilde{\theta}}
\newcommand{\rnorm}{\right \rvert \right \rvert}
\newcommand{\lnorm}{\left \lvert \left \lvert}

\title{Posterior concentration and adaptation of the mixing measure in Dirichlet process mixtures}

\author{Filippo Ascolani \thanks{Duke University, Department of Statistical Science, Durham, NC, United States (filippo.ascolani@duke.edu)}}

\begin{document}
\maketitle

\abstract{
We study the asymptotic properties of the posterior on the latent space for infinite mixtures driven by a Dirichlet process, both in terms of mixing measure and clustering behaviour. In the well-specified regime, where the data are generated by a finite mixture of location densities, we show that the posterior is adaptive to the true number of components $K$: indeed the cumulative mass assigned to weights of the stick-breaking representation beyond the $K$-th one vanishes as $n^{-1/2}$, up to terms growing slower than any polynomial. This also implies a nearly optimal posterior contraction rate for the mixing measure in Wasserstein distance. A remarkable phase transition underlies this result: approximating the mixing measure to any precision finer than $n^{-1/2}$ requires a number of components growing logarithmically with the sample size. We show that this has a profound impact on the clustering behaviour: the number of clusters grows logarithmically, as in the prior case, but the proportion of observations outside the $K$ largest clusters vanishes polynomially fast.
Finally, we turn these results into posterior guarantees for truncation-based approximations: while any truncation with at least $K$ elements recovers the optimal contraction rates for both density and mixing measure, $\mathcal{O}(\log n)$ components are both necessary and sufficient to reproduce the clustering of the exact posterior.}

\section{Introduction}

Mixture models have a long history in Statistics \citep{pearson1894} and are routinely applied in various application domains \citep{ mclachlan1988mixture, lindsay1995mixture, mclachlan2000finite}. In this setting each observation $X \in \mathcal{X}$ is endowed with a likelihood
\begin{equation}\label{eq:general_likelihood}
Pf(x) = \int_{\Theta} f(x \mid \theta)P(\d \theta),
\end{equation}
where $\{f(x \mid \theta)\}_{\theta \in \Theta}$ is a family of probability densities depending on some parameter $\theta \in \Theta$ and $P \in \mathcal{P}(\Theta)$, the class of probability measures on $\Theta$. Then $P$ is called the mixing measure and $Pf$ the corresponding mixture density.

A first use of mixture models is to estimate the "true" data generating density and the mixing measure becomes a useful tool to construct a flexible class of distributions. However in many contexts the mixing measure itself is the main object of interest: indeed $P$ contains the information about the heterogeneity in the population and often yields a useful interpretation. Moreover, in the context of model-based clustering \citep{gormley2023model}, the goal is to partition the data into distinct subgroups: also in this case the main interest lies not in the space of observables $\mathcal{X}$, but in a latent space induced by the mixing measure.

In this paper we focus on the latter two settings within a Bayesian framework, where a prior $\Pi$ is placed on $P$. We consider $\Pi$ to be supported on the space of discrete measures, i.e.
\begin{equation}\label{eq:discrete_measures}
    P(\cdot)\, \overset{\d}{=}\, \sum_{k \geq 1}w_k\delta_{\theta_k}(\cdot)\,,
\end{equation}
where $\w = (w_1, w_2, \dots)$ is a probability vector and $\btheta = (\theta_1, \theta_2, \dots) \in \Theta^\infty$ is a sequence of atoms in $\Theta$. Therefore $\Pi$ can be equivalently described as a joint probability law over the pair $(\w, \btheta)$. In particular we focus on the case where $\Pi$ is the law of a Dirichlet process \citep{ferguson1973bayesian}, which is arguably the most common object in Bayesian nonparametrics: see Section \ref{sec:general} for details. A specific choice of the prior is needed, since some of our results (e.g. Theorems \ref{thm:tail_stick_post} and \ref{thm:number_clusters_post}) show that it is not negligible, even asymptotically: in Section \ref{sec:discussion} we discuss possible extensions to other nonparametric priors.

\subsection*{Main results}

We assume $\mathcal{X} = \Theta = \R^D$ and let $\theta$ be a location parameter, i.e.\ $f(x \mid \theta) = f(x - \theta)$, with $f$ density function on $\R^D$ with respect to the Lebesgue measure. Moreover we let $X_i \simiid P^*f$, where $i = 1,\dots, n$ and the true mixing measure $P^* = \sum_{k = 1}^Kw_k^*\delta_{\theta_k^*}$ has $K < \infty$ support points. Then, if $\Pi$ is the law of a Dirichlet process (see Section \ref{sec:general}), under mild conditions on $f$ our first main result reads
\begin{equation}\label{eq:main_result1}
\Pi\left(\sum_{k > K}w_k > \frac{1}{n^{1/2-\delta}} \mid X_{1:n} \right) \to 0,
\end{equation}
in probability as $n \to \infty$ for every $\delta \in (0, 1/2)$, where $\Pi(\cdot \mid X_{1:n})$ is the posterior distribution of $P$ given $X_{1:n} = (X_1, \dots, X_n)$. This intuitively means that the posterior is adaptive to the true number of components, even if the prior places probability zero on mixing measures with a finite number of support points: indeed atoms beyond the $K$-th one receive a (cumulative) probability no bigger than $n^{-1/2}$, up to terms growing slower than any polynomial. Moreover this implies the following posterior contraction result
\[
\Pi\left(W_1(P^*, P) > \frac{1}{n^{1/2-\delta}} \mid X_{1:n} \right) \to 0,
\]
where $W_1$ denotes the $L^1$-Wasserstein distance. Therefore we show that, at least in a pointwise sense (see Remark \ref{rmk:pointwise_rate} below), Dirichlet process mixtures achieve an almost optimal contraction rate for the mixing measure. We also emphasize that \eqref{eq:main_result1} allows us to obtain an explicit description of the convergence of $P$ to $P^*$: each $\theta_k^*$ is matched (up to vanishing error) with exactly one $\theta_{k'}$, with $k' = 1, \dots, K$. See Theorem \ref{thm:stick_post} and Corollary \ref{crl:convergence_wasserstein} for more precise statements and discussions.

As a second main result, for every $\delta >0$ we also prove that, provided the concentration parameter $\alpha$ is large enough,
\begin{equation}\label{eq:main_result2}
\Pi\left(\sum_{k > \lceil \underline{\beta}\log n\rceil}w_k > \frac{1}{n^{1/2+\delta}}, \sum_{k > \lceil \bar{\beta}\log n\rceil}w_k < \frac{1}{n^{1/2+\delta}} \mid X_{1:n} \right) \to 1
\end{equation}
in probability as $n \to \infty$, for $\underline{\beta} < \bar{\beta}$ depending on $\delta$. Therefore $\mathcal{O}(\log n)$ components are both necessary and sufficient to approximate $P$ with an error smaller than $n^{-1/2}$, qualitatively matching the prior behaviour (Lemma \ref{lemma:weights_priori}). Combining \eqref{eq:main_result2} with \eqref{eq:main_result1} we observe an interesting phase transition: while the first $K$ atoms retain $1-n^{-\gamma}$ portion of the total mass, with $\gamma < 1/2$, any higher order approximation ($\gamma > 1/2$) requires $\mathcal{O}(\log n)$ components. We argue that this phenomenon is due to the nonparametric nature of the model: up to the parametric rate $n^{-1/2}$ the data inform the posterior distribution of $P$, while the prior strongly affects any quantity (e.g.\ the number of clusters, see Theorem \ref{thm:number_clusters_post}) which depends on a finer rate. See Theorem \ref{thm:tail_stick_post} for more details.

As a final application we study truncation methods, where the infinite sum in \eqref{eq:discrete_measures} with weights as in \eqref{GEM} below, is approximated with a finite mixture: this is the basis of popular computational schemes for approximate posterior inference \citep{ishwaran2001gibbs, ishwaran2002approximate}. We prove that a similar phase transition holds: a truncation with $K$ elements contracts to the true density and mixing measure at the same rate of the exact posterior, while $\mathcal{O}(\log n)$ components are necessary and sufficient to match the clustering properties. See Theorem \ref{thm:conv_clustering} for the exact statement.

\subsection*{Structure of the paper}

Section \ref{sec:general} recalls the definition of the Dirichlet process and its main properties used in this paper: the assumptions on the data generating distribution and the kernel $f$ are also stated.  The main technical result (Theorem \ref{thm:evidence_lower_bound}), which provides a tight lower bound on the marginal distribution induced by the model, is discussed in Section \ref{sec:technical}. Section \ref{sec:main_results} contains the main results on the concentration in the latent space:  posterior behaviour of the mixing measure, including the contraction rate in Wasserstein distance (Section \ref{sec:weights}), the phase transition for the mixture weights (Section \ref{sec:transition}) and the implications for clustering (Section \ref{sec:clustering}).  Section \ref{sec:truncation} applies the previous results to provide posterior guarantees for approximations based on truncation, while Section \ref{sec:discussion} discusses future research directions. The proof of Theorem \ref{thm:evidence_lower_bound} is provided in Section \ref{sec:proof_main}, the proofs of all the other results can be found in the Supplementary Material.

\subsection{Related works}\label{sec:related_works}

The convergence rates of mixture models for densities have been extensively studied in both frequentist and Bayesian frameworks \citep{van1996rates, genovese2000rates, ghosal2001entropies}. In case of mixtures driven by a Dirichlet process a series of works proved (nearly) optimal contraction rates under general conditions \citep{ghosal2001entropies, Ghosal2007, scricciolo2011posterior, shen2013adaptive}. See \cite[Chapter 9]{Subhashis2017} for a unifying framework in the case of the Gaussian kernel. 

In the case of the mixing measure, starting from some pioneering works \citep{chen1995optimal, Nguyen2013}, recent papers \citep{ho2016strong, ho2016convergence, heinrich2018strong, guha2021posterior, wei2022convergence,wei2023minimum, bariletto2026convergence} have studied the associated contraction rates in Wasserstein distance. See \cite{nguyen2026optimal} for a recent review. In particular the optimal pointwise convergence rate for the mixing measure is the parametric $n^{-1/2}$, under general settings \citep{heinrich2018strong}: various methods have been shown to achieve such rate, like the minimum distance estimator \citep{heinrich2018strong}, minimum Hellinger distance estimator \citep{ho2020robust}, overfitted mixtures \citep{rousseau2011asymptotic}, mixture of finite mixtures \citep{guha2021posterior}. See also \cite{ishwaran2001bayesian, martin2012convergence} for methods with sub-optimal rates. As far as our knowledge goes, the pointwise rate for Dirichlet process mixtures has not been extensively explored, while the local minimax rate has been shown to be only poly-logarithmic in general \citep{Nguyen2013, guha2021posterior}. Interestingly, mixture models constitute a framework where pointwise and (local) minimax rates can drastically differ \citep{heinrich2018strong, nguyen2026optimal}.

Much less is known for the asymptotic clustering properties. \cite{Miller2013, Miller2014} showed that Dirichlet process mixtures provably overestimate the number of clusters when its concentration parameter (see Section \ref{sec:general}) is fixed: a suitable prior \citep{ascolani2023clustering} or data-dependent choice \citep{ohn2023optimal} of the latter can solve this issue. See \cite{alamichel2024bayesian} for similar inconsistencies for other nonparametric priors. We argue that the lack of results beyond the number of clusters (with the notable exception of \cite{rajkowski2019analysis}) is partially due to the difficulty of relating convergence of the mixing measure to the clustering: indeed the Wasserstein distance is in general too weak for this purpose. We mention \cite{guha2023excess} as a promising work in this direction: the authors use the Orlicz-Wasserstein to prove faster (but still sub-polynomial) convergence in outlier regions of the parameter space for Dirichlet process mixtures of Gaussians.

Finally, some of our results are reminiscent of  \citep{rousseau2011asymptotic}, which studies the asymptotic behaviour of overfitted (but finite) mixtures. Combining their work with this manuscript it becomes clear that in the context of mixture models the choice of the prior can also affect some asymptotic properties, both in the finite- and infinite-dimensional case: see Remarks \ref{rmk:overfitted1} and \ref{rmk:overfitted2} for details and connections with this work. In the case of nonparametric priors it was often suggested that the prior specification might yield a large impact on clustering (e.g.\ \cite[Section 3]{de2013gibbs}), but we are not aware of theoretical results in this direction.

\section{General setting}\label{sec:general}

\subsection{The Dirichlet process}

The Dirichlet process \citep{ferguson1973bayesian} is arguably the most popular object in Bayesian nonparametrics. Following representation \eqref{eq:discrete_measures}, its law can be defined as a probability measure over $(\w, \btheta)$ such that
\begin{equation}\label{GEM}
w_k = v_k\prod_{j = 1}^{k-1}(1-v_j), \quad v_k \simiid \text{Beta}(1, \alpha), \qquad \theta_k \simiid P_0,
\end{equation}
where $\alpha > 0$ is called the concentration parameter and $P_0 \in \mathcal{P}(\Theta)$ the baseline distribution. Then we denote with $P \sim \text{DP}(\alpha, P_0)$ the random measure as in \eqref{eq:discrete_measures} with $(\w, \btheta)$ defined in \eqref{GEM}. Therefore the statistical model
\begin{equation}\label{DPM_model1}
X_i \mid P \simiid Pf(x),  \quad P \sim \text{DP}(\alpha, P_0),
\end{equation}
is called the Dirichlet process mixture model \citep{lo1984class}. It is among the most popular models in Bayesian nonparametrics and its properties have been extensively studied \citep[Chapter 4]{Subhashis2017}.

The construction of the Dirichlet process in \eqref{GEM}, due to \cite{sethuraman1994}, is often called the stick-breaking representation: each $w_k$ is indeed obtained by breaking the remaining part of the unit interval according to the Beta random variable $v_k$. It is also very popular among practitioners, since it leads to natural approximations for tractable posterior inference \citep{ishwaran2001bayesian, ishwaran2002approximate}. In the next lemma we recall some concentration properties used later in the paper: we give the proof in the Supplementary Material for completeness, but analogous results are already known (e.g.\ Theorem 1 in \cite{ishwaran2001bayesian} and Lemma $9.15$ in \cite{Subhashis2017}).
\begin{lemma}\label{lemma:weights_priori}
Let $\Pi$ be the law of $P \sim \text{DP}(\alpha, P_0)$. Then the following holds:
\begin{enumerate}
\item For every $\beta > 0$ we have that
\[
\Pi\left(\sum_{j > \ceil{\beta' \log n}}w_j > \frac{1}{n^\beta} \right) \leq n^{-\alpha\beta},
\]
with $\beta' \geq e\alpha\beta$.

\item For every $\beta > 0$ and $\delta \in (0,1)$, there exists $\beta^* = \beta^*(\alpha, \beta, \delta) > 0$ such that
\[
\Pi\left(\sum_{j > \ceil{\beta' \log n}}w_j <  \frac{1}{n^\beta} \right) \leq n^{-(1-\delta)\alpha \beta}.
\]
with $\beta' < \beta^*$.
\end{enumerate}
\end{lemma}
Since $P$ is almost surely discrete by construction, samples from $P$ will yield ties with positive probability. Therefore model \eqref{DPM_model1} can be equivalently written as
\begin{equation}\label{DPM_model2}
X_i \mid P, c_i \simind f(x\mid \theta_{c_i}), \quad c_i \mid P \simiid \text{Cat}(\w),  \quad P \sim \text{DP}(\alpha, P_0),
\end{equation}
where $\text{Cat}(\w)$ denotes the categorical distribution on $\N$. The latent variables $c_{1:n} = (c_1, \dots, c_n)$, called allocation variables, divide the data into distinct groups and are the objects of interest when performing model-based clustering \citep{gormley2023model}. In the following we denote with $\Phi(\cdot)$ and $\Phi(\cdot \mid X_{1:n})$ respectively the prior and posterior law on $c_{1:n}$ induced by model \eqref{DPM_model2}. A commonly studied functional of $c_{1:n}$ is the number of clusters $K_n$, i.e.\ the number of distinct values in $c_{1:n}$. Its concentration properties are summarized in the next lemma \citep{korwar1973contributions}.
\begin{lemma}\label{lemma: Kn_priori}
Let $\Phi$ be the law on $c_{1:n}$ induced by model \eqref{DPM_model2}. Then the following holds:
\begin{enumerate}
\item Let $\delta > 2$. Then we have that
\[
\Phi\biggl(K_n \geq (1+\delta)\alpha\log\left(1+\frac{n-1}{\alpha} \right) \biggr) \leq  cn^{-\alpha},
\]
where $c = c(\alpha) > 0$.

\item Let $\delta \in( 0, 1)$ such that $\delta -(1-\delta)\log(1-\delta) = 3/4$. Then we have that
\[
\Phi\biggl(K_n \leq (1-\delta)\alpha\log\left(1+\frac{n}{\alpha} \right) \biggr) \leq  cn^{-\frac{3}{4}\alpha},
\]
where $c = c(\alpha) > 0$.
\end{enumerate}
\end{lemma}
Lemmas \ref{lemma:weights_priori} and \ref{lemma: Kn_priori} imply, loosely speaking, that the complexity of model \eqref{DPM_model1} grows logarithmically in the sample size $n$. Indeed an approximation with error vanishing polynomially fast requires $\mathcal{O}(\log n)$ components, and the number of clusters concentrates around $\alpha \log (n)$. One purpose of this paper is to deduce the analog of Lemmas \ref{lemma:weights_priori} and \ref{lemma: Kn_priori} a posteriori.

\subsection{Assumptions}
We suppose to collect a vector of observations $X_{1:n} = (X_1, \dots, X_n)$, with $X_i \in \R^D$. We will make the following assumptions:
\begin{enumerate}
\item[(A1)] There exists $Q \in \mathcal{P}(\R^D)$ such that $X_i \simiid Q$. Moreover $Q(\d x) = P^*f (x)\d x$, where
\[
P^*(\cdot) = \sum_{k = 1}^Kw_k^*\delta_{\theta_{k}^*}(\cdot), \quad K < \infty,
\]
with $(w_1^*, \dots, w_K^*)$ probability vector and $\{ f(x-\theta)\}_{\theta}$ location family of probability densities on $\R^D$.
\item[(A2)] $\Theta \subset \R^D$ is a compact set and $\theta_k^* \in \Theta$ for every $k = 1, \dots, K$.
\item[(A3)] The baseline distribution $P_0$ admits a continuous and bounded density $p_0$ on $\Theta$ with respect to the Lebesgue measure, such that $p_0(\theta_k^*) > 0$ for every $k = 1, \dots, K$.
\end{enumerate}
Assumption (A1) means that model \eqref{DPM_model1} is somewhat well-specified, in the sense that the kernel $f$ matches the one of the data generating mechanism: however the latter is a finite mixture, while almost every realization of $P \sim \text{DP}(\alpha, P_0)$ has infinitely many support points. Assumptions $(A2)-(A3)$ are standard in this setting \citep{Nguyen2013, nguyen2026optimal}.

We also need some smoothness conditions on the kernel density $f$. Denoting with $\left\lVert \cdot \right \rVert$ the Euclidean norm, they read as follows:
\begin{enumerate}
\item[(B1)] $f \in C^3(\R^D)$ and there exists a constant  $F$ such that $0 < f(x) < F$ for every $x \in \R^D$.

\item[(B2)] There exist constants $R > 0$ and $\gamma \geq 0$ such that
\[
\left\lVert \nabla \log f(x) \right \rVert \leq R\left(1+\left\lVert x \right \rVert \right)^\gamma, \quad \left\lVert \nabla^2 \log f(x) \right \rVert \leq R\left(1+\left\lVert x \right \rVert \right)^\gamma
\]
for every $x \in \R^D$ and
\[
\int_{\R^D}e^{M\left\lVert x \right\rVert^\gamma}f(x) \, \d x < \infty
\]
for every $M > 0$.

\item[(B3)] There exist strictly positive constants $C, r$ such that
\[
\left\lVert \nabla^g  f(x) \right \rVert \leq \frac{C}{\left(1+\left\lVert x \right \rVert \right)^{D + r}}
\]
for every $x \in \R^D$ and $g = 1, 2, 3$.
\end{enumerate}
Assumptions $(B1)-(B3)$ are regularity conditions satisfied by common smooth location families, e.g.\ multivariate Gaussian (with $\gamma =1$) and Student's t (with $\gamma = 0$) densities. Assumptions $(A1)-(A3)$ and $(B1)-(B3)$ are sufficient to prove nearly optimal contraction rates for densities, e.g.\ in $L^1$ distance: see Lemma \ref{lm:post_consistency} in the Supplementary Material.

\subsection{The main technical result}\label{sec:technical}

All the later results could be stated as $\Pi(B_n \mid X_{1:n}) \to 0$ (or $\Phi(B_n \mid X_{1:n}) \to 0$) in  probability as $n \to \infty$, where $\{B_n\}_n$ is a sequence of measurable sets. Following a standard strategy, we write
\[
\Pi(B \mid X_{1:n}) = \frac{\int_B \prod_{i = 1}^n\frac{P f(X_i)}{P^* f(X_i)}\Pi(\d P)}{\int \prod_{i = 1}^n\frac{P f(X_i)}{P^* f(X_i)}\Pi(\d P)},
\]
and study separately numerator and denominator. Lower bounds with high probability of the latter are called evidence lower bounds, and they are crucial to deduce contraction rates and other asymptotic properties (e.g.\ \cite[Lemma 8.10]{Subhashis2017}). The next theorem, whose proof is given in Section \ref{sec:proof_main}, provides such lower bound in our setting with explicit dependence on $K$, $D$ and $\alpha$. In the following we use the notation $Q^{(n)}$ to denote the product measure $\otimes_{i = 1}^n Q$ on $\R^{nD}$.
\begin{theorem}\label{thm:evidence_lower_bound}
Under assumptions $(A1)-(A3)$ and $(B1)-(B3)$, for every $\epsilon > 0$ there exists $c := c(\epsilon, Q) > 0$ such that
\[
Q^{(n)}\left(\int \prod_{i = 1}^n\frac{P f(X_i)}{P^* f(X_i)}\Pi(\d P) \geq cn^{-DK/2 - (K-1)/2 -\alpha/2} \right) \geq 1-\epsilon,
\]
for every $n$.
\end{theorem}
The obtained rate is the product of $n^{-DK/2 - (K-1)/2}$, which corresponds to the one of a finite mixture model with $K$ known, and $n^{ -\alpha/2}$, which measures the discrepancy between the nonparametric nature of the model and the (finite) number of parameters needed to describe the data generating mechanism: indeed, the larger $\alpha$ the heavier the tails of the stick-breaking weights (Lemma \ref{lemma:weights_priori}) and the larger the number of clusters (Lemma \ref{lemma: Kn_priori}). Notice that the existence of a (sub-optimal) polynomial rate could be shown with standard techniques (again \cite[Lemma 8.10]{Subhashis2017}) and by Theorem $4.3$ in \cite{hairault2022evidence} there exists $t > 0$ such that
\[
Q^{(n)}\left(\int \prod_{i = 1}^n\frac{P f(X_i)}{P^* f(X_i)}\Pi(\d P) \geq cn^{-DK/2 - (K-1)/2 -t} \right) \to 0,
\]
as $n \to \infty$. Therefore the rate obtained in Theorem \ref{thm:evidence_lower_bound} is tight as a function of $K$ and $D$. Moreover, its main value is to make the dependence of the rate on $\alpha$ explicit: this will allow us to combine it with the prior concentration results stated in Lemmas \ref{lemma:weights_priori} and \ref{lemma: Kn_priori}.
\begin{remark}\label{rmk:overfitted1}
The results of Theorem \ref{thm:evidence_lower_bound} are reminiscent of the case of overfitted mixtures, where $P \sim \Pi_{\bar{K}}$ is defined as follows
\[
P \overset{\d}{=} \sum_{j = 1}^{\bar{K}}w_j\delta_{\theta_j}, \quad (w_1, \dots, w_{\bar{K}}) \sim \text{Dir}\left(\alpha_1, \dots, \alpha_{\bar{K}}\right), \quad \theta_j \simiid P_0,
\]
with $\bar{K} > K$. Then from the proof of Theorem $1$ in \cite{rousseau2011asymptotic} we can deduce that
\[
Q^{(n)}\left(\int \prod_{i = 1}^n\frac{P f(X_i)}{P^* f(X_i)}\Pi_{\bar{K}}(\d P) \geq cn^{-DK/2 - (K-1)/2 -\sum_{j = K + 1}^{\bar{K}}\alpha_j/2} \right) \to 1,
\]
as $n \to \infty$ for some $c > 0$. Therefore the evidence lower bound has the same form with the concentration parameter $\alpha$ replaced by $\sum_{j = K+1}^{\bar{K}}\alpha_j$: both quantities measure how much the model is overfitting the true data generating mechanism. This similarity is particularly striking, since $\Pi_{\bar{K}}$ and $\Pi$ are probability measures on, respectively, a finite- and infinite-dimensional space: see Remark \ref{rmk:overfitted2} for a deeper connection between \cite{rousseau2011asymptotic} and our work.
\end{remark}
The proof of Theorem \ref{thm:evidence_lower_bound} is given in Section \ref{sec:proof_main} and relies on a Taylor expansion of the stick-breaking weights $(v_1, \dots, v_K)$ around the ones induced by the true vector $\w^*$ in (A1). The main difficulty comes from the fact that the coefficients and the remainder term of such expansion depend on the (infinitely many) weights and atoms of $P$: we control their behavior borrowing tools from empirical process theory \citep{van2023weak}, see Lemmas \ref{lemma:glivenko-cantelli} and \ref{lemma:donsker} below.

\section{Main results on concentration in the latent space}\label{sec:main_results}

\subsection{Asymptotic behaviour of the mixing measure}\label{sec:weights}
The first main result concerns the probability assigned, a posteriori, to the elements of the stick-breaking representation beyond the $K$-th ones.
\begin{theorem}\label{thm:stick_post}
Under assumptions $(A1)-(A3)$ and $(B1)-(B3)$, for every $\delta > 0$ it holds that
\[
\Pi\left(\sum_{k > K}w_k > \frac{1}{n^{1/2-\delta}} \mid X_{1:n} \right) \to 0,
\]
as $n \to \infty$ in $Q^{(\infty)}$-probability.
\end{theorem}
This crucially implies that the Dirichlet process mixture model is adaptive to the true number of components: indeed, up to the parametric rate $n^{-1/2}$, the probability mass is concentrated on the first $K$ components of the stick-breaking representation. This is not obvious in general, first of all because a finite mixture can be always represented with a larger number of components; moreover the stick-breaking weights in \eqref{GEM} are only ordered in expectation and a priori $\mathcal{O}(\log n)$ components are required to cover $1-n^{-\beta}$ of the mass for every $\beta > 0$ (Lemma \ref{lemma:weights_priori}). We provide below a sketch of the proof (inspired by Section $D.2$ in \cite{hairault2022evidence}), whose more technical aspects are deferred to the Supplementary Material.
\begin{proof}[Sketch of the proof of Theorem \ref{thm:stick_post}]
By Lemma \ref{lm:post_consistency} in the Supplementary Material, the result is equivalent to
\[
\Pi\left(\left\{\sum_{j > K}w_j > \frac{1}{n^{1/2-\delta}}\right\} \cap \left\{||Pf - P^*f||_1 \leq \frac{(\log n)^q}{\sqrt{n}} \right\} \mid X_{1:n} \right) \to 0,
\]
as $n \to \infty$ in $Q^{(\infty)}$-probability, where $||\cdot||_1$ denotes the $L^1$ distance and $q > 0$ is a fixed constant. Moreover, by Lemma \ref{lm:lower_bound_consistency} in the Supplementary Material (which follows directly from Theorem $3.10$ in \cite{gassiat2014local}), there exist $\{B_k\}_k$ arbitrarily small neighborhoods of $\{\theta_k^*\}_k$ and a constant $m > 0$ such that
\begin{equation}\label{eq:lower_bound_main}
\begin{aligned}
||Pf &- P^*f||_1 \\
&\geq m\biggl\{ P(B_0) + \sum_{k = 1}^K\left\lvert P(B_k) - w_k ^* \right\rvert + \sum_{k = 1}^K\left\lvert\left\lvert \int_{B_k}(\theta-\theta_k^*)P(\d \theta) \right\rvert\right\rvert + \frac{1}{2}\sum_{k = 1}^K \int_{B_k}\left\lvert\left\lvert\theta-\theta_k^* \right\rvert\right\rvert^2P(\d \theta) \biggr\},
\end{aligned}
\end{equation}
with $B_0 = \left(\cup_{k = 1}^KB_k\right)^c$. Combined with Theorem \ref{thm:evidence_lower_bound}, the result then follows by proving that there exists $\beta > 0$ such that
\begin{equation}\label{eq:to_prove_main}
\begin{aligned}
\Pi\biggl(&\left\{\sum_{j > K}w_j > \frac{1}{n^{1/2-\delta}}\right\} \cap \left\{P(B_0) + \sum_{k = 1}^K\left\lvert P(B_k) - w_k ^* \right\rvert \leq \frac{(\log n)^q}{m\sqrt{n}} \right\}\\
& \quad\cap \left\{\sum_{k = 1}^K\left\lvert\left\lvert \int_{B_k}(\theta-\theta_k^*)P(\d \theta) \right\rvert\right\rvert +\frac{1}{2}\sum_{k = 1}^K\int_{B_k}\left\lvert\left\lvert\theta-\theta_k^* \right\rvert\right\rvert^2P(\d \theta)\leq \frac{(\log n)^q}{m\sqrt{n}} \right\}\biggr) < n^{-(DK + K + \alpha-1)/2 -\beta},
\end{aligned}
\end{equation}
for $n$ large enough. This is proved in Proposition \ref{prop:prior_bound} of the Supplementary Material.
\end{proof}
We stress that the two main steps of the proof are given by the lower bound \eqref{eq:lower_bound_main}, which comes from \cite{gassiat2014local} and is restricted to location mixtures, and the upper bound on the prior probability in \eqref{eq:to_prove_main}. The latter is greatly simplified by the tractability of the Dirichlet process: for example by \eqref{GEM} it easily follows that
\begin{equation}\label{eq:independence_normalization}
\sum_{k = 1}^Kw_k \delta_{\theta_k} \quad \text{and} \quad \sum_{k>K}\left(\frac{w_k}{1-\sum_{j = 1}^Kw_j} \right)\delta_{\theta_k}
\end{equation}
are independent random measures. Beyond the technical aspects, the result in Theorem \ref{thm:stick_post} greatly relies on the stick-breaking representation \eqref{GEM}. For the sake of intuition, consider the case $K = 1$ with $w_1^* = 1$ and $\theta_1^* = \theta^*$. Approximating $P^*f$ with only the first element requires concentration of $v_1$ (around $1$) and $\theta_1$ (around $\theta^*$). Using only the second component, instead, means concentration for $v_1$ (around $0$), $v_2$ (around $1$) and $\theta_2$ (around $\theta^*$); similarly, a combination of the first two components constrains $v_2$, $\theta_1$ and $\theta_2$. Therefore we expect the posterior distribution to favour the first, more parsimonious, setting: the proof above can be seen as a rigorous extension of this intuitive argument to all possible configurations that approximate the true density.

As an important consequence of Theorem \ref{thm:stick_post}, we can deduce the contraction rate for the mixing measure using the Wasserstein distance 
\[
W_1(P, P') = \inf_{\gamma \in \mathcal{C}(P, P')} \, \int \left \lvert \left \lvert \theta-\theta' \right \rvert \right \rvert \, \gamma(\d \theta, \d \theta'), \quad P, P' \in \mathcal{P}(\Theta),
\]
where $\mathcal{C}(P, P')$ denotes the space of couplings between $P$ and $P'$. The rate is a consequence of a stronger result, which provides the exact way in which $P^*$ is approximated. This is formalized in the next corollary, where we use $\mathcal{S}_K$ to denote the space of permutations of $[K] = (1, \dots, K)$ (i.e.\ the set of bijective functions from $[K]$ to itself).
\begin{corollary}\label{crl:convergence_wasserstein}
Under assumptions $(A1)-(A3)$ and $(B1)-(B3)$, for every $\delta > 0$ it holds that
\[
\Pi \left(\exists \text{ $\sigma \in \mathcal{S}_K$ s.t. } \left \lvert\left \lvert \theta_k^*-\theta_{\sigma(k)} \right \rvert \right \rvert \leq \frac{1}{n^{1/2-\delta}} \text{ and } \left \lvert w_k^*-w_{\sigma(k)} \right \rvert \leq \frac{1}{n^{1/2-\delta}} \text{ for } k = 1, \dots, K \mid X_{1:n}\right) \to 1,
\]
as $n \to \infty$ in $Q^{(\infty)}$-probability. In particular
\[
\Pi\left(W_1(P^*, P) > \frac{1}{n^{1/2-\delta}} \mid X_{1:n} \right) \to 0,
\]
as $n \to \infty$ in $Q^{(\infty)}$-probability.
\end{corollary}
\begin{remark}\label{rmk:pointwise_rate}
The contraction rate above holds for a fixed $P^*$ and it is therefore pointwise: it is also nearly equal to the optimal pointwise rate $n^{-1/2}$ \citep{heinrich2018strong, nguyen2026optimal}, up to terms growing slower than any polynomial. It is known that the local minimax rate is much slower (e.g.\, \cite[Theorem 3.2]{heinrich2018strong} for the case when an upper bound on $K$ is known), see \cite{nguyen2026optimal} for a discussion on the different types of rates. As far as our knowledge goes, this is the first result proving almost optimal pointwise rate for the mixing measure in the case of models driven by a Dirichlet process: previous (minimax) rates were only poly-logarithmic \citep{Nguyen2013, guha2021posterior}.
\end{remark}

\begin{remark}\label{rmk:overfitted2}
Corollary \ref{crl:convergence_wasserstein}
is the analog of Theorem $1$ in \cite{rousseau2011asymptotic}, which shows that in overfitted mixture models the posterior concentrates on configurations with no more than $K$ components, again up to a remainder of order $n^{-1/2}$. However in the latter case the prior parameters of the Dirichlet distribution are required to be all smaller than $D/2$, while our result does not require any condition on the concentration parameter $\alpha$: this is coherent with the fact that the Dirichlet process can be seen as the limit of a sequence of $M$-variate symmetric Dirichlet distributions with parameters $\alpha/M$ as $M $ diverges \citep{ishwaran2002exact}. See Remark \ref{rmk:lower_bound_alpha} for a further discussion on the possible impact of $\alpha$ asymptotically.
\end{remark}
Corollary \ref{crl:convergence_wasserstein} implies that the Dirichlet process mixture model achieves an almost optimal pointwise rate for the mixing measure. Within the class of Bayesian nonparametric priors, this was previously known for the mixture of finite mixtures model \citep[Theorem 3.1]{guha2021posterior}, which explicitly incorporates the number of components as a parameter. Moreover we show that, at least asymptotically, each $\theta_k^*$ is approximated by exactly one atom of $P$; this is a much stronger result, which has deep consequences on the clustering properties of the model. 

\subsection{The phase transition}\label{sec:transition}
If Theorem \ref{thm:stick_post} considers the stick-breaking weights up to error $n^{-1/2}$, in the next theorem we deal with any finer approximation.
\begin{theorem}\label{thm:tail_stick_post}
Under assumptions $(A1)-(A3)$ and $(B1)-(B3)$, for every $\delta > 0$ there exists $\alpha^* = \alpha^*(D, K, \delta) > 0$ such that for every $\alpha > \alpha^*$ it holds that
\begin{equation}\label{eq:result_tail_to_cite}
\Pi\left(\sum_{k > \lceil \underline{\beta}\log n\rceil}w_k > \frac{1}{n^{1/2+\delta}}, \sum_{k > \lceil \bar{\beta}\log n\rceil}w_k < \frac{1}{n^{1/2+\delta}} \mid X_{1:n} \right) \to 1,
\end{equation}
as $n \to \infty$ in $Q^{(\infty)}$-probability, for some $0 < \underline{\beta} < \bar{\beta}$.

\end{theorem}
This result implies that $\mathcal{O}(\log n)$ components are both necessary and sufficient to capture at least $1-n^{-\gamma}$ of the mass of $P$, with $\gamma > 1/2$. 
\begin{remark}\label{rmk:lower_bound_alpha}
For the proof of Theorem \ref{thm:tail_stick_post} we need the concentration parameter $\alpha$ to be fixed larger than some $\alpha^*$, which depends on $D$, $K$ and $\delta$; moreover, according to our calculations, $\alpha^*$ diverges as $\delta$ approaches zero. This phenomenon seems unavoidable with our proof technique, but it is unclear to us whether a constraint on $\alpha$ is truly needed: as already mentioned in Remark \ref{rmk:overfitted2}, even in the finite-dimensional case of overfitted mixtures \citep{rousseau2011asymptotic} the hyperparameters of the prior play a role also asymptotically.
\end{remark}

Combining Theorems \ref{thm:stick_post} and \ref{thm:tail_stick_post} we deduce an interesting phase transition: the first $K$ elements of the stick-breaking representation account for the total mass up to nearly $n^{-1/2}$, while any smaller remainder requires a logarithmic (in $n$) number of components. We argue that this transition is due to the interaction between the data and the prior: up to the parametric rate $n^{-1/2}$ the data inform the mixing measure, while any finer approximation is strongly driven by the prior. Indeed the statement of Theorem \ref{thm:tail_stick_post} is qualitatively similar to the one of Lemma \ref{lemma:weights_priori}, so that prior and posterior behaviour match. Moreover the proof of Theorem \ref{thm:tail_stick_post}, given in the Supplementary Material, consists exactly in showing that the complement of the event in \eqref{eq:result_tail_to_cite} yields prior probability vanishing faster than the rate in Theorem \ref{thm:evidence_lower_bound}.

\subsection{Asymptotic clustering behaviour}\label{sec:clustering}

In this section we focus on the asymptotic properties of the allocation variables $c_{1:n}$, with respect to the posterior $\Phi(\cdot \mid X_{1:n})$ induced by model \eqref{DPM_model2}. The next theorem provides matching upper and lower bounds for the number of clusters.
\begin{theorem}\label{thm:number_clusters_post}
Under assumptions $(A1)-(A3)$ and $(B1)-(B3)$, there exists $\alpha^* = \alpha^*(D, K) > 0$ such that for every $\alpha > \alpha^*$ we have that
\[
\Phi\biggl((1+\bar{\delta})\alpha\log(n) \geq K_n \geq (1-\underline{\delta})\alpha\log(n) \mid X_{1:n}\biggr) \to 1,
\]
as $n \to \infty$ in $Q^{(\infty)}$-probability, where $\underline{\delta} \in (0,1)$ and $\bar{\delta} > 0$.
\end{theorem}
This means that the number of clusters grows logarithmically in $n$, as in the prior case (see Lemma \ref{lemma: Kn_priori}). The result is also coherent with the previous theorems: indeed, $\mathcal{O}(\log n)$ components are needed to account for $1-n^{-1}$ of the mixing mass by Theorem \ref{thm:tail_stick_post}. We believe that this phenomenon arises, intuitively, since the number of clusters is largely driven by the prior. In Lemma \ref{lm:preliminary_cn} of the Supplementary Material we also show that each $c_i$ must be smaller than $\ceil{\beta \log n}$, for some $\beta > 0$; therefore the components of the mixture with at least one allocation variable must belong to the first $\mathcal{O}(\log n)$ elements of the stick-breaking construction. We also stress that $\bar{\delta}$ and $\underline{\delta}$ can be chosen independently from $\alpha$: this implies that the aymptotic growth of the number of clusters depends on $\alpha$ and therefore on the exact prior specification.

Theorem \ref{thm:number_clusters_post} also complements the results in \cite{Miller2013, Miller2014}, which showed that $K_n > K$ almost surely as $n \to \infty$, by providing the exact rate of growth of the number of clusters. Even if such result was described as a lack of consistency for clustering and therefore as a negative feature of Dirichlet process mixture models, with the next theorem we argue that the situation is more nuanced.
\begin{theorem}\label{thm:residual_clustering}
Under assumptions $(A1)-(A3)$ and $(B1)-(B3)$, for every $\delta >0 $ it holds that
\[
\Phi\biggl( \frac{1}{n}\sum_{i = 1}^n\mathbbm{1}_{\{ c_i > K\}} > \frac{1}{n^{1/2-\delta}} \mid X_{1:n}\biggr) \to 0,
\]
as $n \to \infty$ in $Q^{(\infty)}$-probability.
\end{theorem}
Theorem \ref{thm:residual_clustering} shows that the proportion of observations allocated to components after the $K$-th one vanishes polynomially fast. We deduce that the posterior clustering features $K$ big clusters, each corresponding to one component of $P^*$, and a logarithmic (in $n$) number of clusters whose overall proportion tends to zero. Therefore we argue that the posterior clustering induced by model \eqref{DPM_model2} is actually very close to the one induced by the true mixing measure $P^*$, up to clusters with vanishing relevance: the inconsistency noted by \cite{Miller2014, Miller2017}, then, is due to the very peculiar nature of the number of groups as a summary of the clustering.

\section{Posterior guarantees for approximations based on truncation}\label{sec:truncation}

Working directly with the posterior of model \eqref{DPM_model1} can be challenging, since $P$ is an infinite dimensional object: a popular solution \citep{ishwaran2001gibbs, ishwaran2002approximate} is therefore to consider a finite dimensional approximation. Given a truncation level $N \geq 1$, the latter is defined as
\begin{equation}\label{eq:trunc_model}
X_i \mid \tilde{P} \simiid \tilde{P}f(x), \quad \tilde{P} \overset{\text{d}}{=} \sum_{k = 1}^{N+1}\tilde{w}_k\delta_{\tilde{\theta}_k},
\end{equation}
where $\tilde{w}_k$ and $\tilde{\theta}_{k'}$ are as in \eqref{GEM}, with $k = 1, \dots, N$ and $k' = 1, \dots, N+1$, while $\tilde{w}_{N+1} = 1-\sum_{k = 1}^N\tilde{w}_k$. Coherently with the previous notation, we will use $\tilde{\Pi}_N(\cdot)$ and $\tilde{\Pi}_N(\cdot \mid X_{1:n})$ to denote the associated prior and posterior distributions on $\mathcal{P}(\Theta)$.

Thus model \eqref{eq:trunc_model} approximates \eqref{DPM_model1} by replacing the infinite dimensional $P$ with its truncated version $\tilde{P}$. The latter procedure is usually justified based on results similar to Lemma \ref{lemma:weights_priori}, by showing that the $L^1$ difference of the associated marginal distributions decays exponentially fast in the truncation level $N$ (see e.g.\ Theorems $1$ and $2$ in \cite{ishwaran2001gibbs} or Theorem $1$ and Corollary $1$ in \cite{ishwaran2002approximate}). However, as far as our knowledge goes, no guarantees are known a posteriori: in this section we use the results in Theorems \ref{thm:stick_post} and \ref{thm:tail_stick_post} to fill this gap. 

First of all, under assumptions $(A1)-(A3)$ it is clear that any $N \geq K$ would suffice to recover a nearly optimal contraction rate for densities. The next corollary, which we state for completeness and can be also deduced with similar arguments as in \cite{rousseau2011asymptotic} (see also \cite[Section $3.1$]{guha2021posterior}), shows that the same holds for convergence of the mixing measure.
\begin{corollary}\label{crl:convergence_wasserstein_trunc}
Under assumptions $(A1)-(A3)$ and $(B1)-(B3)$, for every $N \geq K$ and $\delta > 0$ it holds that
\[
\tilde{\Pi}_N\left(W_1(P^*, \tilde{P}) > \frac{1}{n^{1/2-\delta}} \mid X_{1:n} \right) \to 0,
\]
as $n \to \infty$ in $Q^{(\infty)}$-probability.
\end{corollary}
We investigate then the clustering behaviour. Denoting with $\tilde{c}_{1:n}$ the vector of allocation variables induced by model \eqref{eq:trunc_model} and with $\tilde{\Phi}_N(\cdot)$ its prior law, it is natural to quantify the difference in terms of clustering as a suitable distance between the posterior distribution on the allocation variables  of the original model $\Phi(\cdot \mid X_{1:n})$ and $\tilde{\Phi}_N(\cdot \mid X_{1:n})$. The next theorem provides an explicit result in the case of the total variation distance $\left \lvert \left \lvert \cdot \right\rvert \right\rvert_{TV}$.

\begin{theorem}\label{thm:conv_clustering}
Under assumptions $(A1)-(A3)$ and $(B1)-(B3)$, there exists $\beta' = \beta'(D, K) > 0$ such that for every $N \geq \lceil \beta' \log n \rceil$ we have that
\[
\left \lvert \left \lvert \Phi \left(c_{1:n} \mid X_{1:n} \right)-\tilde{\Phi}_N \left(\tilde{c}_{1:n} \mid X_{1:n}  \right) \right\rvert \right\rvert_{TV} \to 0,
\]
in $Q^{(\infty)}$- probability as $n \to \infty$. Moreover, there exist $\alpha^* = \alpha^*(D, K)$ and $\beta = \beta(\alpha^*)$ such that if $\alpha > \alpha^*$ and $N < \lceil \beta \log n \rceil$ we have that
\[
\left \lvert \left \lvert \Phi \left(c_{1:n} \mid X_{1:n} \right)-\tilde{\Phi}_N \left(\tilde{c}_{1:n} \mid X_{1:n}  \right) \right\rvert \right\rvert_{TV} \to 1,
\]
in $Q^{(\infty)}$- probability as $n \to \infty$.
\end{theorem}
Therefore a truncation level $N = \mathcal{O}(\log n)$ is both necessary and sufficient to recover the clustering behaviour of the original posterior. We argue that this phenomenon is strongly connected to the phase transition discussed in Section \ref{sec:transition}: clustering depends on finer properties of the mixing measure, which might not be preserved with a $n^{-1/2}$ approximation error. 

As practical takeaways, Corollary \ref{crl:convergence_wasserstein_trunc} and Theorem \ref{thm:conv_clustering} show that truncation methods can be an efficient way to approximate the infinite-dimensional posterior under our assumptions, since $\mathcal{O}(\log n)$ parameters are sufficient to recover (arguably) all the features of interest. Moreover, an appropriate choice of the truncation level deeply depends on the purpose of the approximation: a fixed value $N$ is enough to achieve the same contraction rates of the original model for the density and the mixing measure, if $N \geq K$, but a matching clustering behaviour requires $N$ to grow with the sample size.

\section{Discussion}\label{sec:discussion}

In this article we provide results on the posterior behaviour of Dirichlet process mixture models when the data are generated by a well-specified finite mixture: in this setting we can show that each component of the true mixing measure is matched by exactly one atom of the stick-breaking representation, up to sub-polynomial terms. Moreover we observe an interesting phase transition in the posterior distribution of the weights: those explicit results allow us to study the asymptotic clustering behaviour, for which little is known in the nonparametric framework.

A first natural generalization is given by priors beyond the Dirichlet process. We believe that similar results could hold for stick-breaking priors \citep{ishwaran2001bayesian}, where the weights are as in \eqref{GEM} with $V_k \overset{\text{ind.}}{\sim} \text{Beta}(a_k, b_k)$; the choice $a_k = 1-\sigma$ and $b_k = \alpha +\sigma k$, with $\sigma \in [0,1)$, leads to the well-known Pitman-Yor process \citep{pitmanyor1997}. We conjecture that Theorem \ref{thm:stick_post} holds and therefore the mass still concentrates on the first $K$ components. Theorems \ref{thm:tail_stick_post} and \ref{thm:number_clusters_post} should instead be modified according to the different asymptotic prior behaviour: for example we expect $K_n = \mathcal{O}(n^\sigma)$ in the Pitman-Yor case, instead of $\mathcal{O}(\log n)$ we prove here. However non-trivial technical challenges could arise especially in the proof of Theorem \ref{thm:stick_post}, which is significantly simplified by specific prior properties of the Dirichlet process; in this perspective a promising avenue could be given by priors based on completely random measures \citep{barrios2013modeling}, which retain much analytical tractability. Connected to this point, it is also well-known \citep{ascolani2023clustering, ohn2023optimal} that the posterior can be very sensitive to the choice of the concentration parameter $\alpha$, on which a suitable hyper-prior is therefore often placed; extending our results to this hierarchical specification could be of interest.

Moreover, assumption $(A1)$ of Section \ref{sec:general} could be relaxed to allow the data to be generated by an \emph{infinite} mixture. We believe that many results of this paper could still hold under suitable assumptions on the tail behaviour of the true weights, which should be coherent with the prior results in Lemma \ref{lemma: Kn_priori}. This could also provide useful tools to study the extension to data generated from (super)smooth distributions, which can be well-approximated by location mixtures with a growing number of components (see e.g.\ \cite[Section 9.4]{Subhashis2017}): we are not aware of any posterior results on clustering or truncated schemes in this context.

Finally, the choice of kernels beyond the location family is an interesting avenue for future research. The main difficulty is to recover the lower bound \eqref{eq:lower_bound_main}, which follows from Theorem $3.10$ in \cite{gassiat2014local} and requires a location family. Extending the latter to other classes of kernels, e.g.\ being strongly identifiable \citep{Nguyen2013, guha2021posterior}, is a significant technical challenge.

\section{Proof of Theorem \ref{thm:evidence_lower_bound}}\label{sec:proof_main}

Before the actual proof, we need to introduce some preliminary notation and results.

\paragraph{A Taylor expansion of the log-likelihood.}
Denote with $v^*_{1:K} = (v_1^*, \dots, v_K^*)$ the vector of positive numbers such that $v_k^* \in [0, 1]$ and
\[
v_1^* = w_1^*, \quad v_k^*\prod_{i = 1}^{k-1}(1-v_i^*) = w_k^*, \quad k = 2, \dots, K.
\]
Notice that by construction $v^*_{K} = 1$. Moreover, let $\vs = (v_1, v_2, \dots)$, $\w = (w_1, w_2, \dots)$ and $\btheta = (\theta_1, \theta_2, \dots)$ such that
\[
v_k \in [0, 1], \quad w_k = v_k\prod_{j = 1}^{k-1}(1-v_j), \quad \theta_k \in \R^D,
\]
for every $k \geq 1$. For a fixed $n$ and $x_{1:n} = (x_1, \dots, x_n)$, with $x_i \in \R^D$, we define 
\[
f_{\vs, \btheta}(x) = \sum_{k \geq 1}w_kf(x -\theta_k), \quad l_{\vs, \btheta}(x) = \log \left(f_{\vs, \btheta}(x) \right) = \log \left(\sum_{k \geq 1}w_kf(x -\theta_k) \right)
\]
and
\begin{equation}\label{eq:def_loglikelihood}
l_{\vs, \btheta}(x_{1:n}) = \sum_{i = 1}^nl_{\vs, \btheta}(x_i) = \sum_{i = 1}^n\log \left(\sum_{k \geq 1}w_kf(x_i -\theta_k) \right).
\end{equation}
We consider a Taylor expansion of $l_{\vs, \btheta}(x_{1:n})$ with respect to $\theta_{1:K}$ and $v_{1:K}$ around $\theta^*_{1:K}$ and $v^*_{1:K}$, respectively: notice that this is well-defined, since for every $x_{1:n}$ the function $l_{\vs, \btheta}(x_{1:n})$ is differentiable for $v_K \in (0, 1+\delta)$ with $\delta > 0$ depending on $x_{1:n}$. Therefore there exist suitable coefficients
\[
A_{n, k}(\vs, \btheta) \in \R^D, \quad B_{n, k}(\vs, \btheta) \in \R
\]
and
\[
\quad C_{n, k, k'}(\vs, \btheta) \in \R^{D \times D}, \quad D_{n, k, k'}(\vs, \btheta) \in \R^{D \times 1}, \quad E_{n, k, k'}(\vs, \btheta)\in \R,
\]
with $k, k' = 1, \dots, K$ such that
\begin{equation}\label{eq:taylor_expansion}
\begin{aligned}
   l_{\vs, \btheta}&(x_{1:n}) = \\
   &l_{\vs^*, \btheta^*}(x_{1:n}) + \sum_{k = 1}^KA_{n,k}(\vs^*, \btheta^*)(\theta_k - \theta_k^*) + \sum_{k = 1}^KB_{n,k}(\vs^*, \btheta^*)(v_k - v_k^*) + R^{(\vs, \btheta)}_n(\hat{\vs}, \hat{\btheta}),
\end{aligned}
\end{equation}
where $(\vs^*, \btheta^*)$ and $(\hat{\vs}, \hat{\btheta})$ are such that
\begin{equation}\label{eq:definition_star}
\vs^*_{1:K} = v^*_{1:K}, \quad \btheta^*_{1:K} = \theta^*_{1:K}, \quad v_k^* = v_k, \quad \theta^*_k = \theta_k,
\end{equation}
for every $k > K$ and
\[
\hat{\vs}_{1:K} = t\vs_{1:K} + (1-t)v^*_{1:K}, \quad \hat{\btheta}_{1:K} = t\btheta_{1:K} + (1-t)\theta^*_{1:K}, \quad \hat{v}_k = v_k, \quad \hat{\theta}_k = \theta_k,
\]
with $t \in (0,1)$ depending on $x_{1:n}$, and
\begin{equation}\label{eq:remainder}
\begin{aligned}
    R^{(\vs, \btheta)}_n(\hat{\vs}, \hat{\btheta}) = \frac{1}{2}\sum_{k, k' = 1}^K(\theta_k - \theta_k^*)^T&C_{n,k, k'}(\hat{\vs}, \hat{\btheta})(\theta_{k'} - \theta_{k'}^*)+\frac{1}{2}\sum_{k, k' = 1}^K(\theta_k - \theta_k^*)^TD_{n,k, k'}(\hat{\vs}, \hat{\btheta})(v_{k'} - v_{k'}^*)\\
& +\frac{1}{2}\sum_{k, k' = 1}^KE_{n, k, k'}(\hat{\vs}, \hat{\btheta})(v_{k} - v_k^*)(v_{k'} - v_{k'}^*).
\end{aligned}
\end{equation}
We give below the explicit expression of the entries of $A, B, C, D, E$, which correspond to the first and second derivatives of $l_{\vs, \btheta}(x_{1:n})$ with respect to $\vs_{1:K}$ and $\btheta_{1:K}$. In the following we also use the notation
\[
\partial_d f(x) = \left. \frac{\partial f(y)}{\partial y_d} \right|_{y = x}, \quad \partial^2_{d d'} f(x) = \left. \frac{\partial^2 f(y)}{\partial y_d\partial y_{d'}} \right|_{y = x}
\]
for $d, d' = 1, \dots, D$.

As regards the first derivatives, for $k = 1, \dots, K$ and $d = 1, \dots, D$ we have that
\begin{equation}\label{eq:definition_A}
A_{n, k, d}(\vs, \btheta) = \frac{\partial l_{\vs, \btheta}(x_{1:n})}{\partial \theta_{kd}} = \sum_{i = 1}^n\frac{\partial l_{\vs, \btheta}(x_{i})}{\partial \theta_{kd}} = -w_k\sum_{i = 1}^n\frac{\partial_df(x_i - \theta_k)}{f_{\vs, \btheta}(x_i)}
\end{equation}
and
\begin{equation}\label{eq:definition_B}
\begin{aligned}
B_{n, k}(\vs, \btheta) &= \frac{\partial l_{\vs, \btheta}(x_{1:n})}{\partial v_{k}} = \sum_{i = 1}^n\frac{\partial l_{\vs, \btheta}(x_{i})}{\partial v_k} \\
&= \left( 1- \sum_{j = 1}^{k-1}w_j\right)\sum_{i = 1}^n\frac{f(x_i - \theta_k) - \sum_{j > k}\bar{w}_jf(x_i -\theta_j)}{f_{\vs, \btheta}(x_i)},
\end{aligned}
\end{equation}
with $\bar{w}_j = v_j\prod_{i = k+1}^{j-1}(1-v_i)$ for $j > k$.

As regards the second derivatives, for $k, k' = 1, \dots, K$ and $d, d' = 1, \dots, D$ we have that
\begin{equation}\label{eq:definition_C}
C_{n, k, k', d, d'}(\vs, \btheta) = \frac{\partial^2 l_{\vs, \btheta}(x_{1:n})}{\partial \theta_{kd} \partial \theta_{k'd'}} = 
\begin{cases}
    w_k\sum_{i = 1}^n\frac{\partial_{dd'}f(x_i - \theta_k)}{f_{\vs, \btheta}(x_i)} -\sum_{i = 1}^n\left[\frac{\partial l_{\vs, \btheta}(x_{i})}{\partial \theta_{kd}} \right] \left[\frac{\partial l_{\vs, \btheta}(x_{i})}{\partial \theta_{kd'}} \right] \quad \text{if } k = k'\\
    \\
    -\sum_{i = 1}^n\left[\frac{\partial l_{\vs, \btheta}(x_{i})}{\partial \theta_{kd}} \right] \left[\frac{\partial l_{\vs, \btheta}(x_{i})}{\partial \theta_{k'd'}} \right] \quad \text{if } k \neq k'
\end{cases}
\end{equation}
and
\begin{equation}\label{eq:definition_D}
D_{n, k, k', d}(\vs, \btheta) = \frac{\partial^2 l_{\vs, \btheta}(x_{1:n})}{\partial \theta_{kd}\partial v_{k'}} = 
\begin{cases}
\frac{1}{v_k}\frac{\partial l_{\vs, \btheta}(x_{1:n})}{\partial \theta_{kd}}-\sum_{i = 1}^n\left[\frac{\partial l_{\vs, \btheta}(x_{i})}{\partial \theta_{kd}} \right] \left[\frac{\partial l_{\vs, \btheta}(x_{i})}{\partial v_k} \right] \quad \text{if } k = k'\\
\\
-\sum_{i = 1}^n\left[\frac{\partial l_{\vs, \btheta}(x_{i})}{\partial \theta_{kd}} \right] \left[\frac{\partial l_{\vs, \btheta}(x_{i})}{\partial v_{k'}} \right] \quad \text{if } k < k'\\
\\
-\frac{1}{1-v_{k'}}\frac{\partial l_{\vs, \btheta}(x_{1:n})}{\partial \theta_{kd}}-\sum_{i = 1}^n\left[\frac{\partial l_{\vs, \btheta}(x_{i})}{\partial \theta_{kd}} \right] \left[\frac{\partial l_{\vs, \btheta}(x_{i})}{\partial v_{k'}} \right] \quad \text{if } k > k'
\end{cases}
\end{equation}
and
\begin{equation}\label{eq:definition_E}
E_{n, k, k'}(\vs, \btheta) = \frac{\partial^2 l_{\vs, \btheta}(x_{1:n})}{\partial v_k\partial v_{k'}} = 
\begin{cases}
-\sum_{i = 1}^n\left[ \frac{\partial l_{\vs, \btheta}(x_{i})}{\partial v_k}\right]^2 \quad \text{if } k = k'\\
\\
-\frac{1}{1-v_{k}}\frac{\partial l_{\vs, \btheta}(x_{1:n})}{\partial v_k}-\sum_{i = 1}^n\left[\frac{\partial l_{\vs, \btheta}(x_{i})}{\partial v_k} \right] \left[\frac{\partial l_{\vs, \btheta}(x_{i})}{\partial v_{k'}} \right] \quad \text{if } k < k'\\
\\
-\frac{1}{1-v_{k'}}\frac{\partial l_{\vs, \btheta}(x_{1:n})}{\partial v_{k'}}-\sum_{i = 1}^n\left[\frac{\partial l_{\vs, \btheta}(x_{i})}{\partial v_k} \right] \left[\frac{\partial l_{\vs, \btheta}(x_{i})}{\partial v_{k'}} \right] \quad \text{if } k > k'
\end{cases}
\end{equation}

\paragraph{Results on Glivenko-Cantelli and Donsker classes.}

Let $\sH$ be a subset of the space of measurable functions $h \, : \, \R^D \, \to \, \R$ such that
\[
Qh^2 = \int h^2(x) Q(\d x) < \infty, \quad \text{for every $h \in \sH$}
\]
and
\[
\sup_{h \in \sH} \, |h(x) - Qh| < \infty, \quad \text{for every $x \in \R^D$}.
\]
Given $X_i \simiid Q$, with $i = 1, \dots, n$ we denote
\[
P_nh = \frac{1}{n}\sum_{i = 1}^nh(X_i), \quad G_nh = \sqrt{n}\left(P_nh - Qh \right).
\]
Following usual practice, we say that $\sH$ is \emph{Q-Glivenko-Cantelli} if
\[
\sup_{h \in \sH} \, |P_nh - Qh| \to 0,
\]
almost surely as $n \to \infty$. Similarly, we say that $\sH$ is \emph{Q-Donsker} if the sequence of processes $\left\{G_nh \, \mid \, h \in \mathcal{H} \right\}$ converges weakly to a tight limit process in $\ell^{\infty}(\sH)$, the space of bounded functions from $\sH$ to $\R$. For more details, see Chapter 19 in \cite{van2000asymptotic} and \cite{van2023weak}.

\begin{lemma}\label{lemma:glivenko-cantelli}
Let $\Theta \subset \R^D$ be compact and
\[
\begin{aligned}
\sH_1 = \biggl\{\left[\frac{f(x-\theta_1)}{f(x-\theta_2)}\right]^p \, &\mid \, \theta_1, \theta_2 \in \Theta, p = 1,2 \biggr\}, \quad \sH_2 = \left\{\left[\frac{\partial_df(x-\theta_1)}{f(x-\theta_2)}\right]^p \, \mid \, \theta_1, \theta_2 \in \Theta, p = 1,2 \right\}
\end{aligned}
\]
and
\[
\sH_3 = \left\{\frac{\partial_{dd'}f(x-\theta_1)}{f(x-\theta_2)} \, \mid \, \theta_1, \theta_2 \in \Theta \right\},
\]
with $d, d' = 1, \dots, D$. Under assumptions $(B1) - (B3)$, we have that $\sH_1$, $\sH_2$ and $\sH_3$ are Q-Glivenko-Cantelli.
\end{lemma}
\begin{proof}
The general principle, which follows e.g.\ by Example $19.8$ in \cite{van2000asymptotic}, is that $\sH = \left\{h_\theta(x) \, \mid \, \theta \in \Theta \right\}$ is Q-Glivenko-Cantelli if $\theta \, \to \, h_\theta(x)$ is continuous for every $x$ and there exists $H(x)$ such that
\begin{equation}\label{eq:condition_GC}
|h_\theta(x)| \leq H(x), \quad \int H(x) P^*f(x) \, \d x < \infty.
\end{equation}
We start by $\sH_1$ and notice that by $(B1)$ and the Mean Value Theorem we have that
\[
\begin{aligned}
p \left\lvert \log f(x-\theta_1) -  \log f(x-\theta_2)\right]\rvert &\leq p\left\lVert \theta_1 - \theta_2 \right\rVert \sup_{\theta \in \Theta}\, \left\lVert \nabla \log f(x-\theta) \right\rVert \leq pB s(x),
\end{aligned}
\]
where $B = \text{diam}(\Theta)$, which is finite by compactness of $\Theta$, and 
\begin{equation}\label{eq:definition_s}
s(x) = \sup_{\theta \in \Theta}\, \left\lVert \nabla \log f(x-\theta) \right\rVert.
\end{equation}
Thus, in order to prove \eqref{eq:condition_GC} it suffices to show that
\[
\int e^{pB s(x)}P^*f(x) \,d x = \sum_{k = 1}^Kw_k^*\int e^{pB s(x+\theta_k^*)}f(x) \,d x < \infty,
\]
which is implied by
\begin{equation}\label{eq:exp_moment1}
\int e^{pB s(x+\eta)}f(x) \,d x < \infty,
\end{equation}
for every fixed $\eta \in \R^D$. By assumption $(B2)$ and simple calculations we get that
\[
\begin{aligned}
s(x + \eta) &= \sup_{\theta \in \Theta}\, \left\lVert \nabla \log f(x+\eta-\theta) \right\rVert \leq R\sup_{\theta \in \Theta}\, \left( 1+ \left\lVert x + \eta - \theta \right\rVert \right)^\gamma\\
& \leq R\sup_{\theta \in \Theta}\, \left( 1+ \left\lVert x + \eta - \theta \right\rVert \right)^\gamma \leq R\left( 1+ \left\lVert x  \right\rVert +\rho\right)^\gamma,
\end{aligned}
\]
with $\rho = \sup_{\theta \in \Theta}\,\left\lVert \eta - \theta \right\rVert < \infty$. Let $c_\gamma$ such that $(a+b)^\gamma \leq c_\gamma(a^\gamma+b^\gamma)$, so that
\[
s(x + \eta) \leq Rc_\gamma \left[(1+\rho)^\gamma + \left\lVert x  \right\rVert^\gamma\right],
\]
which implies
\[
\int e^{pB s(x+\eta)}f(x) \,d x \leq e^{pBRc_\gamma(1+\rho)}\int e^{pBRc_\gamma\left\lVert x  \right\rVert^\gamma}f(x) \,d x
\]
and \eqref{eq:exp_moment1} follows again by $(B2)$.

As regards $\sH_2$ notice that
\[
\begin{aligned}
\left \lvert\frac{\partial_df(x-\theta_1)}{f(x-\theta_2)}\right\rvert & = \left\lvert\frac{\partial_df(x-\theta_1)}{f(x-\theta_1)}\right\rvert\frac{f(x-\theta_1)}{f(x-\theta_2)}\\
& = \left\lvert\partial_d \log f(x-\theta_1)\right\rvert\frac{f(x-\theta_1)}{f(x-\theta_2)} \leq s(x)e^{Bs(x)},
\end{aligned}
\]
with $s(x)$ as in \eqref{eq:definition_s}. By the elementary inequality $t e^{ct} \leq e^{(c+2)t}$, for every $t,c \geq 0$, we have that
\[
\left \lvert\frac{\partial_df(x-\theta_1)}{f(x-\theta_2)}\right\rvert^p \leq e^{p(B+2)s(x)}
\]
and \eqref{eq:condition_GC} follows as in the previous point.

As regards $\sH_3$ notice that
\[
\frac{\partial_{dd'}f(x)}{f(x)} = \partial_{dd'}\log f(x) + \left[\frac{\partial_{d}\log f(x)}{f(x)} \right]\left[\frac{\partial_{d'}\log f(x)}{f(x)} \right]
\]
and therefore
\[
\begin{aligned}
\left \lvert\frac{\partial_{dd'}f(x-\theta_1)}{f(x-\theta_2)}\right\rvert & = \left\lvert\frac{\partial_{dd'}f(x-\theta_1)}{f(x-\theta_1)}\right\rvert\frac{f(x-\theta_1)}{f(x-\theta_2)}\\
& = \left\lvert\partial_{dd'} \log f(x-\theta_1)\right\rvert\frac{f(x-\theta_1)}{f(x-\theta_2)} + \left\lvert\partial_{d} \log f(x-\theta_1)\right\rvert\left\lvert\partial_{d'} \log f(x-\theta_1)\right\rvert\frac{f(x-\theta_1)}{f(x-\theta_2)}\\
&\leq s_2(x)e^{Bs(x)} + s^2(x)e^{Bs(x)},
\end{aligned}
\]
with $s_2(x) = \sup_{\theta \in \Theta}\, \left\lVert \nabla^2 \log f(x-\theta) \right\rVert$. Again by $(B2)$, we get that
\[
\begin{aligned}
s_2(x + \eta) &= \sup_{\theta \in \Theta}\, \left\lVert \nabla^2 \log f(x+\eta-\theta) \right\rVert \leq  R\left( 1+ \left\lVert x  \right\rVert +\rho\right)^\gamma,
\end{aligned}
\]
which implies that, proceeding as before, there exists $C'> 0$ such that
\[
\int \left[s_2(x)e^{Bs(x)} + s^2(x)e^{Bs(x)} \right]P^*f(x) \, \d x \leq e^{C'}\int e^{C'\left\lVert x  \right\rVert^\gamma}f(x) \,d x
\]
and \eqref{eq:condition_GC} follows as in the previous points.
\end{proof}

\begin{lemma}\label{lemma:donsker}
Let $\Theta \subset \R^D$ be compact and
\[
\begin{aligned}
\sH = \biggl\{\frac{f(x-\theta)}{P^*f(x)} \, &\mid \, \theta \in \Theta \biggr\}.
\end{aligned}
\]
Under assumptions $(B1) - (B3)$ we have that $\sH$ is Q-Donsker.
\end{lemma}
\begin{proof}
The general principle, which follows e.g.\ by Example $19.7$ in \cite{van2000asymptotic}, is that $\sH = \left\{h_\theta(x) \, \mid \, \theta \in \Theta \right\}$ is Q-Donsker if there exists $m(x)$ such that
\begin{equation}\label{eq:condition_D}
\left\lvert h_{\theta_1}(x) - h_{\theta_2}(x) \right\rvert \leq \left\lVert \theta_1 - \theta_2 \right\rVert m(x), \quad \int m(x) P^*f(x) \, \d x < \infty,
\end{equation}
for every $x \in \R^D$ and $\theta_1, \theta_2 \in \Theta$. By $(B1)$ and the Mean Value Theorem we have that
\[
\begin{aligned}
\left\lvert \frac{f(x-\theta_1)}{P^*f(x)} -  \frac{f(x-\theta_2)}{P^*f(x)}\right]\rvert &\leq \left\lVert \theta_1 - \theta_2 \right\rVert m(x),
\end{aligned}
\]
with 
\[
m(x) = \frac{\sup_{\theta \in \Theta}\, \left\lVert \nabla f(x-\theta) \right\rVert}{P^*f(x)}.
\]
By assumption $(B3)$ and simple calculations, there exists $C' > 0$ such that
\[
\begin{aligned}
\int m(x) P^*f(x) \, \d x &= \int \sup_{\theta \in \Theta}\, \left\lVert \nabla f(x-\theta) \right\rVert \, \d x \leq \int \frac{C'}{\left(1+\left\lVert x \right \rVert \right)^{D + r}} \, \d x < \infty,
\end{aligned}
\]
and therefore \eqref{eq:condition_D} follows.
\end{proof}

\paragraph{Preliminary results}
Combining the materials developed in the previous paragraphs, we can prove the following corollaries.
\begin{corollary}\label{cor:upper_bound_taylor1}
For every $\delta > 0$ there exists $C = C(\delta, P^*)$ such that
\[
Q^{(n)}\left( \sup_{\vs^*, \btheta^*} \,  \frac{1}{\sqrt{n}}\left\lvert A_{n,k, d}(\vs^*, \btheta^*) \right \rvert \leq C\right) \geq 1 -\delta, \quad Q^{(n)}\left( \sup_{\vs^*, \btheta^*} \,  \frac{1}{\sqrt{n}}\left\lvert B_{n,k}(\vs^*, \btheta^*) \right \rvert \leq C\right) \geq 1 -\delta,
\]
for every $k = 1, \dots, K$ and $d = 1, \dots, D$ and $n \geq 1$, with $(\vs^*, \btheta^*)$ as in \eqref{eq:definition_star}.
\end{corollary}
\begin{proof}
First of all notice that
\[
A_{n, k, d}(\vs^*, \btheta^*) = -w^*_k\sum_{i = 1}^n\frac{\partial_df(X_i - \theta^*_k)}{P^*f(X_i)},
\]
for every $(\vs^*, \btheta^*)$. Moreover we have that
\[
\E \left[\frac{\partial_df(X - \theta^*_k)}{P^*f(X)}  \right] = 0
\]
and
\[
\begin{aligned}
\E \left[\left(\frac{\partial_df(X - \theta^*_k)}{P^*f(X)} \right)^2 \right] &\leq \frac{1}{(w_k^*)^2}\int \left[ \partial_d \log f(x-\theta_k^*)\right]^2P^*f(x)\, \d x\\
&\leq \frac{R^2}{(w_k^*)^2}\sum_{k' = 1}^Kw_{k'}^*\int \left( 1+\left\lVert x-\theta_k^* \right\rVert\right)^{2\gamma}f(x-\theta_{k'}^*)\, \d x <\infty,
\end{aligned}
\]
by applying $(B2)$ twice. The result for $A_{n, k, d}(\vs^*, \btheta^*)$ then follows by the Central Limit Theorem.

As regards the second point, notice that
\[
\begin{aligned}
\left \lvert\sum_{i = 1}^n\frac{f(x_i - \theta^*_k) - \sum_{j > k}\bar{w}_jf(x_i -\theta_j)}{P^*f(x_i)}\right \rvert &\leq \left \lvert\sum_{i = 1}^n\left(\frac{f(x_i - \theta^*_k)}{P^*f(x_i)}-1\right)\right \rvert+ \left \lvert\sum_{i = 1}^n\left( \frac{\sum_{j > k}\bar{w}_jf(x_i -\theta_j)}{P^*f(x_i)}-1\right)\right \rvert\\
& \leq 2\sup_{\theta \in \Theta} \,\left \lvert\sum_{i = 1}^n\left(\frac{f(x_i - \theta)}{P^*f(x_i)}-1\right)\right \rvert
\end{aligned},
\]
which implies
\[
\sup_{\vs^*, \btheta^*} \,  \frac{1}{\sqrt{n}}\left\lvert B_{n,k}(\vs^*, \btheta^*) \right \rvert \leq 2\sup_{\theta \in \Theta} \,\left \lvert\sum_{i = 1}^n\left(\frac{f(X_i - \theta)}{P^*f(X_i)}-1\right)\right \rvert
\]
Since
\[
\E \left[\frac{f(X - \theta)}{P^*f(X)}-1 \right] = 0,
\]
by Lemma \ref{lemma:donsker} we have that
\[
\sup_{\theta \in \Theta} \left \lvert \frac{1}{\sqrt{n}}\sum_{i = 1}^n\left(\frac{f(X_i - \theta)}{P^*f(X_i)}-1\right) \right \rvert \to \sup_{\theta \in \Theta}\left\lvert \G \left(\frac{f(\cdot - \theta)}{P^*f(\cdot)}\right)  \right\rvert
\]
weakly as $n \to \infty$, where $\G$ is a tight Gaussian process in $\ell^{\infty}(\sH)$. Therefore the right hand side is finite and the result follows.
\end{proof}
Let now $\epsilon' > 0$ and define
\begin{equation}\label{eq:definition_S}
S_{\epsilon'} = \left\{(\vs, \btheta) \, \mid \, |v_k - v_k^*|\leq \epsilon', \left\lVert\theta_{k} - \theta_{k}^*\right\rVert\leq \epsilon', v_{k'} \in [0,1] \text{ for every } k = 1, \dots, K \text{ and } k' \geq 1 \right\}.
\end{equation}
\begin{corollary}\label{cor:upper_bound_taylor2}
Let
\[
\epsilon' < \min_{k = 1, \dots, K} \,v_k^* \quad \text{ and } \quad \epsilon' < 1-\max_{k = 1, \dots, K-1} \,v_k^*.
\]
Then, for every $\delta > 0$ there exist $C = C(\delta, \theta^*)$ such that
\[
Q^{(n)}\left( \sup_{(\vs, \btheta) \in S_{\epsilon'}} \,  \frac{1}{n}\left\lvert C_{n,k,k',d, d'}(\vs, \btheta) \right \rvert \leq C\right) \geq 1 -\delta, \quad Q^{(n)}\left( \sup_{(\vs, \btheta) \in S_{\epsilon'}} \,  \frac{1}{n}\left\lvert D_{n,k, k', d}(\vs, \btheta) \right \rvert \leq C\right) \geq 1 -\delta,
\]
and
\[
Q^{(n)}\left( \sup_{(\vs, \btheta) \in S_{\epsilon'}} \,  \frac{1}{n}\left\lvert E_{n,k, k'}(\vs, \btheta) \right \rvert \leq C\right) \geq 1 -\delta
\]
for every $k = 1, \dots, K$ and $d = 1, \dots, D$ and $n \geq 1$.
\end{corollary}
\begin{proof}
By the constraints on $\epsilon$ we immediately get that
\[
\sup_{(\vs, \btheta) \in S_{\epsilon'}} \, \sup_{k = 1, \dots, K}\, \frac{1}{v_k} < \infty \quad \text{and} \quad \sup_{(\vs, \btheta) \in S_{\epsilon'}} \, \sup_{k = 1, \dots, K-1}\, \frac{1}{1-v_k} < \infty.
\]
Then by \eqref{eq:definition_C}, \eqref{eq:definition_D} and \eqref{eq:definition_E}, to prove the result it suffices to show that there exists $K = K(\delta, P^*)$ such that for every $k,k' = 1, \dots, K$ and $d, d' = 1, \dots, D$ it holds
\begin{equation}\label{eq:toshow_bounded1}
 Q^{(n)}\left(\sup_{(\vs, \btheta) \in S_{\epsilon'}} \,\frac{1}{n}\sum_{i = 1}^n\left \lvert\frac{\partial l_{\vs, \btheta}(X_{i})}{\partial \theta_{kd}}\right\rvert  \leq K\right) \geq 1-\delta
\end{equation}
and
\begin{equation}\label{eq:toshow_bounded2}
Q^{(n)}\left(\sup_{(\vs, \btheta) \in S_{\epsilon'}} \,\frac{1}{n}\sum_{i = 1}^n\left \lvert\frac{\partial l_{\vs, \btheta}(X_{i})}{\partial v_k}\right\rvert  \leq K\right) \geq 1-\delta
\end{equation}
and
\begin{equation}\label{eq:toshow_bounded3}
Q^{(n)}\left(\sup_{(\vs, \btheta) \in S_{\epsilon'}} \,\frac{1}{n}\sum_{i = 1}^n\left \lvert\frac{\partial_{dd'} f(X_{i}-\theta_k)}{f_{\vs, \btheta}(X_i)}\right\rvert  \leq K\right) \geq 1-\delta
\end{equation}
and
\begin{equation}\label{eq:toshow_bounded4}
Q^{(n)}\left(\sup_{(\vs, \btheta) \in S_{\epsilon'}} \,\frac{1}{n}\sum_{i = 1}^n\left \lvert\left[\frac{\partial l_{\vs, \btheta}(X_{i})}{\partial \theta_{kd}} \right] \left[\frac{\partial l_{\vs, \btheta}(X_{i})}{\partial \theta_{k'd'}}\right]\right\rvert \leq K\right) \geq 1-\delta
\end{equation}
and
\begin{equation}\label{eq:toshow_bounded5}
Q^{(n)}\left(\sup_{(\vs, \btheta) \in S_{\epsilon'}} \,\frac{1}{n}\sum_{i = 1}^n\left \lvert\left[\frac{\partial l_{\vs, \btheta}(X_{i})}{\partial v_k} \right] \left[\frac{\partial l_{\vs, \btheta}(X_{i})}{\partial v_{k'}}\right]\right\rvert \leq K\right) \geq 1-\delta
\end{equation}
and
\begin{equation}\label{eq:toshow_bounded6}
Q^{(n)}\left(\sup_{(\vs, \btheta) \in S_{\epsilon'}} \,\frac{1}{n}\sum_{i = 1}^n\left \lvert\left[\frac{\partial l_{\vs, \btheta}(X_{i})}{\partial \theta_{kd}} \right] \left[\frac{\partial l_{\vs, \btheta}(X_{i})}{\partial v_{k'}}\right]\right\rvert  \leq K\right) \geq 1-\delta.
\end{equation}

As regards \eqref{eq:toshow_bounded1}, \eqref{eq:toshow_bounded2} and \eqref{eq:toshow_bounded3}, notice that
\[
\begin{aligned}
\sup_{(\vs, \btheta) \in S_{\epsilon'}} \,\frac{1}{n}\sum_{i = 1}^n\left \lvert\frac{\partial l_{\vs, \btheta}(X_{i})}{\partial \theta_{kd}}\right\rvert &= \sup_{(\vs, \btheta) \in S_{\epsilon'}} \,\frac{1}{n}\sum_{i = 1}^n\left \lvert\frac{ \partial_{d} f(X_{i}-\theta_k)}{f_{\vs, \btheta}(X_{i})}\right\rvert\\
&\leq \sup_{\theta_1, \theta_2 \in \Theta} \,\frac{K}{n}\sum_{i = 1}^n\left \lvert\frac{\partial_{d} f(X_{i}-\theta_1)}{f(X_i - \theta_2)}\right\rvert
\end{aligned}
\]
and
\[
\begin{aligned}
\sup_{(\vs, \btheta) \in S_{\epsilon'}} \,\frac{1}{n}\sum_{i = 1}^n\left \lvert\frac{\partial l_{\vs, \btheta}(X_{i})}{\partial v_k}\right\rvert &= \sup_{(\vs, \btheta) \in S_{\epsilon'}} \,\frac{1}{n}\sum_{i = 1}^n\left \lvert\frac{ f(X_{i}-\theta_k) - \sum_{j>K}\bar{w}_jf(X_i-\theta_j)}{f_{\vs, \btheta}(X_{i})}\right\rvert\\
&\leq \sup_{\theta_1, \theta_2 \in \Theta} \,\frac{2K}{n}\sum_{i = 1}^n\left \lvert\frac{ f(X_{i}-\theta_1)}{f(X_i - \theta_2)}\right\rvert
\end{aligned}
\]
and
\[
\sup_{(\vs, \btheta) \in S_{\epsilon'}} \,\frac{1}{n}\sum_{i = 1}^n\left \lvert\frac{\partial_{dd'} f(X_{i}-\theta_k)}{f_{\vs, \btheta}(X_i)}\right\rvert \leq \sup_{\theta_1, \theta_2 \in \Theta} \,\frac{K}{n}\sum_{i = 1}^n\left \lvert\frac{\partial_{dd'} f(X_{i}-\theta_1)}{f(X_i - \theta_2)}\right\rvert,
\]
so that \eqref{eq:toshow_bounded1}, \eqref{eq:toshow_bounded2} and \eqref{eq:toshow_bounded3} follow directly by Lemma \ref{lemma:glivenko-cantelli}.

As regards \eqref{eq:toshow_bounded4}, by H\"older inequality we have that
\[
\begin{aligned}
    \frac{1}{n}\sum_{i = 1}^n\left \lvert\left[\frac{\partial l_{\vs, \btheta}(X_{i})}{\partial \theta_{kd}} \right] \left[\frac{\partial l_{\vs, \btheta}(X_{i})}{\partial \theta_{k'd'}}\right]\right\rvert \leq \sqrt{\frac{1}{n}\sum_{i = 1}^n\left[\frac{\partial l_{\vs, \btheta}(X_{i})}{\partial \theta_{kd}}  \right]^2}\sqrt{\frac{1}{n}\sum_{i = 1}^n\left[\frac{\partial l_{\vs, \btheta}(X_{i})}{\partial \theta_{k'd'}}  \right]^2},
\end{aligned}
\]
and therefore \eqref{eq:toshow_bounded4} follows as before by applying Lemma \ref{lemma:glivenko-cantelli}. Similarly holds for \eqref{eq:toshow_bounded5} and \eqref{eq:toshow_bounded6}.
\end{proof}

We can finally provide the proof of Theorem \ref{thm:evidence_lower_bound}.
\begin{proof}[Proof of Theorem \ref{thm:evidence_lower_bound}]
We can rewrite the integral in terms of the laws of $(\w, \btheta)$ in \eqref{GEM}, i.e.
\[
\begin{aligned}
\int \prod_{i = 1}^n\frac{P f(X_i)}{P^* f(X_i)}\Pi(\d P) &= \int \prod_{i = 1}^n\frac{\sum_{j \geq 1}w_jf(X_i-\theta_j)}{\sum_{k = 1}^Kw_k^*f(X_i-\theta^*_k)}G_0(\d \vs)P_0^{(\infty)}(\d \btheta)\\
& = \int e^{l_{\vs, \btheta}(X_{1:n})-l_{\vs^*, \btheta^*}(X_{1:n})}G_0(\d \vs)P_0^{(\infty)}(\d \btheta),
\end{aligned}
\]
where $l_{\vs, \btheta}(x_{1:n})$ is as in \eqref{eq:def_loglikelihood} and $G_0$ denotes the law of $\vs$ induced by \eqref{GEM}, i.e.\ $G_0 = \otimes_{k \geq 1} \, \text{Beta}(1, \alpha)$.  By \eqref{eq:taylor_expansion} we can then write
\[
\int \prod_{i = 1}^n\frac{P f(X_i)}{P^* f(X_i)}\Pi(\d P) = \int e^{\sum_{k = 1}^KA_{n,k}(\vs^*, \btheta^*)(\theta_k - \theta_k^*) + \sum_{k = 1}^KB_{n,k}(\vs^*, \btheta^*)(v_k - v_k^*) + R^{(\vs, \btheta)}_n(\hat{\vs}, \hat{\btheta})}G_0(\d \vs)P_0^{(\infty)}(\d \btheta),
\]
with $A_{n,k}$, $B_{n,k}$ and $R_n$ as in \eqref{eq:definition_A}, \eqref{eq:definition_B} and \eqref{eq:remainder} respectively. Denote now with $T_n$ the set $S_\epsilon$ as in \eqref{eq:definition_S}, with $\epsilon = n^{-1/2}$, and consider now the change of variables
\[
\tilde{\theta}_k = \sqrt{n}(\theta_k - \theta_k^*),  \quad \text{and} \quad \tilde{v}_k = \sqrt{n}(v^*_k - v_k), \quad k = 1, \dots, K
\]
and $(\tilde{v}_k, \tilde{\theta}_k) = (v_k, \theta_k)$ for $k > K$, so that 
\[
\begin{aligned}
\int &\prod_{i = 1}^n\frac{P f(X_i)}{P^* f(X_i)}\Pi(\d P) \geq \int_{T_n} e^{\sum_{k = 1}^KA_{n,k}(\vs^*, \btheta^*)(\theta_k - \theta_k^*) + \sum_{k = 1}^KB_{n,k}(\vs^*, \btheta^*)(v_k - v_k^*) + R^{(\vs,\btheta)}_n(\hat{\vs}, \hat{\btheta})}G_0(\d \vs)P_0^{(\infty)}(\d \btheta)\\
&\geq \frac{r^{2K-1}}{n^{DK/2+K/2}}\int_{T} e^{\sum_{k = 1}^K\left\{\frac{A_{n,k}(\vs^*, \btheta^*)}{\sqrt{n}}\tilde{\theta_k} - \frac{B_{n,k}(\vs^*, \btheta^*)}{\sqrt{n}}\tilde{v}_k\right\}+\tilde{R}^{(\tilde{\vs},\tilde{\btheta})}_n(\hat{\vs}, \hat{\btheta})}\left(1-v_K^*+\frac{\tilde{v}_K}{\sqrt{n}}\right)^{\alpha-1}G_0(\d \tilde{\vs}_{-(1:K)})P_0^{(\infty)}(\d \tilde{\btheta}),
\end{aligned}
\]
where $\tilde{\vs}_{-(1:K)}$ denotes the sequence $\tilde{\vs}$ without the first $K$ entries, and
\[
T = \left\{\left(\tilde{\vs}, \tilde{\btheta}\right) \, \mid \, \tilde{v}_k \in[0,1],  ||\tilde{\theta}_{k'}|| \leq 1, \text{ for } k \geq 1, k' = 1, \dots, K \right\}
\]
and $r >0$ is such that $p_0(\theta) \geq r$ for $\theta$ in an arbitrary neighborhood of $\theta_k^*$, for $k = 1, \dots, K$ and Beta$(v \mid 1, \alpha) \geq r$ for $v$ in an arbitrary neighborhood of $v_k^*$, for $k = 1, \dots, K-1$. Notice that such $r$ must exist by $(A3)$. Moreover, with a little abuse of notation, we still use $(\hat{\vs}, \hat{\btheta})$ to denote the vector
\[
\hat{\btheta}_{1:K} = \theta^*_{1:K}+\frac{t}{\sqrt{n}}\tilde{\btheta}_{1:K}, \quad \hat{\vs}_{1:K} = v^*_{1:K}-\frac{t}{\sqrt{n}}\tilde{\vs}_{1:K}, \quad \hat{v}_k = \tilde{v}_k, \quad \hat{\theta}_k = \tilde{\theta}_k,
\]
for some $t\in (0,1)$ and for every $k > K$. Finally, we use the notation
\[
\tilde{R}^{(\tilde{\vs},\tilde{\btheta})}_n(\hat{\vs}, \hat{\btheta}) = R^{\left(\vs^*-\tilde{\vs}/\sqrt{n},\btheta^*+\tilde{\btheta}/\sqrt{n}\right)}_n(\hat{\vs}, \hat{\btheta}),
\]
with $R_n^{(\vs, \btheta)}(\hat{\vs}, \hat{\btheta})$ as in \eqref{eq:remainder}.

By construction $v_K^* = 1$, which implies that
\[
\left(1-v_K^*+\frac{\tilde{v}_K}{\sqrt{n}}\right)^{\alpha-1} = \frac{\left(\tilde{v}_K\right)^{\alpha-1}}{n^{\alpha/2-1/2}}
\]
and therefore
\begin{equation}\label{eq:technical_lower_bound}
\begin{aligned}
\int &\prod_{i = 1}^n\frac{P f(X_i)}{P^* f(X_i)}\Pi(\d P) \\
&
\begin{aligned}
\geq  \frac{r^{2K-1}}{n^{DK/2+(K-1)/2+\alpha/2}}\int_{T} \left(\tilde{v}_K\right)^{\alpha-1}e^{\sum_{k = 1}^K\left\{\frac{A_{n,k}(\vs^*, \btheta^*)}{\sqrt{n}}\tilde{\theta_k} - \frac{B_{n,k}(\vs^*, \btheta^*)}{\sqrt{n}}\tilde{v}_k\right\}+\tilde{R}^{(\tilde{\vs},\tilde{\btheta})}_n(\hat{\vs}, \hat{\btheta})}G_0(\d \tilde{\vs}_{-(1:K)})P_0^{(\infty)}(\d \tilde{\btheta}).
\end{aligned}
\end{aligned}
\end{equation}
Fix now $\epsilon > 0$, let $C = C(\epsilon) > 0$ and consider the events
\[
\Omega_{1,n}(C) = \left\{X_{1:n} \, \mid \, \sup_{\vs^*, \btheta^*}\,\frac{|A_{n,k}|}{\sqrt{n}} \leq C, \sup_{\vs^*, \btheta^*}\,\frac{||B_{n,k}||}{\sqrt{n}} \leq C \right\},
\]
with $(\vs^*, \btheta^*)$ as in \eqref{eq:definition_star}, and
\[
\Omega_{2,n}(C) = \left\{X_{1:n} \, \mid \,  \sup_{(\hat{\vs}, \hat{\btheta}) \in S_{\epsilon'}, (\tilde{\vs}, \tilde{\btheta}) \in T}\,\left\lvert \tilde{R}^{(\tilde{\vs}, \tilde{\btheta})}_n(\hat{\vs}, \hat{\btheta})\right\rvert  \leq C \right\},
\]
with $\epsilon'$ as in Corollary \eqref{cor:upper_bound_taylor2}.Then if $X_{1:n} \in \Omega_{1,n}(C) \cap \Omega_{2,n}(C)$, for some $C > 0$, by \eqref{eq:technical_lower_bound} it is easy to deduce that
\[
\int \prod_{i = 1}^n\frac{P f(X_i)}{P^* f(X_i)}\Pi(\d P) \geq cn^{-DK/2-(K-1)/2 - \alpha/2},
\]
for $c > 0$ depending on $C$ and $r$. Thus, it suffices to prove that there exists $C > 0$ such that
\begin{equation}\label{eq:prob_stat}
Q^{(n)}\left(X_{1:n} \in \Omega_{1,n}(C) \cap \Omega_{2,n}(C) \right) \geq 1-\epsilon.
\end{equation}
By Corollary \ref{cor:upper_bound_taylor1} it follows directly that there exists $C_1 > 0$ such that
\[
Q^{(n)}\left(X_{1:n} \not\in \Omega_{1,n}(C_1) \right) \leq \frac{\epsilon}{2}.
\]
Analogously, by definition of $R_n(\vs, \btheta)$ in \eqref{eq:remainder} we have that
\[
\begin{aligned}
\tilde{R}^{(\tilde{\vs}, \tilde{\btheta})}_n(\hat{\vs}, \hat{\btheta}) = \frac{1}{2n}\sum_{k, k' = 1}^K(\tilde{\theta}_k)^T&C_{n,k, k'}(\hat{\vs}, \hat{\btheta})\tilde{\theta}_{k'}-\frac{1}{2n}\sum_{k, k' = 1}^K(\tilde{\theta}_k)^TD_{n,k, k'}(\hat{\vs}, \hat{\btheta})\tilde{v}_{k'}\\
& +\frac{1}{2n}\sum_{k, k' = 1}^KE_{n, k, k'}(\hat{\vs}, \hat{\btheta})\tilde{v}_{k}\tilde{v}_{k'}
\end{aligned}
\]
and therefore by Corollary \ref{cor:upper_bound_taylor2} there exists $C_2 > 0$ such that
\[
Q^{(n)}\left(X_{1:n} \not\in \Omega_{2,n}(C_2) \right) \leq \frac{\epsilon}{2}.
\]
Thus \eqref{eq:prob_stat} follows with $C = C_1 + C_2$.
\end{proof}

\paragraph{Acknowledgements.} The author would like to thank Professor Surya Tokdar for
invaluable conversations and comments. The author was partially supported by the National Institute of Health (grant ID 1R01-GM163225-01).






\bibliographystyle{chicago}
\bibliography{DPM_bib.bib}

\newpage
\begin{appendix}
\section{Proofs}
\subsection{Proof of Lemma \ref{lemma:weights_priori}}
\begin{proof}[Proof of Lemma \ref{lemma:weights_priori}]
Consider point $1$. Fix $N \in \N$ and $R \in (0,1)$. Notice that by \eqref{GEM} we have that $-\log\left(\sum_{j > N}w_j \right) =- \sum_{j = 1}^N\log (1-v_j) \sim \text{Gamma}(N, \alpha)$, so that for every $\lambda > 0$ by Markov's inequality it holds
\[
\begin{aligned}
\Pi\left(\sum_{j > N}w_j >R \right) &= \Pi\left(e^{\lambda \sum_{j = 1}^N\log(1-v_j)} > e^{\lambda \log (R)} \right)\\
& \leq e^{-\lambda \log (R)}\E\left[e^{\lambda \sum_{j = 1}^N\log(1-v_j)} \right] = e^{-\lambda \log (R)}\left( 1+\frac{\lambda}{\alpha}\right)^{-N}.
\end{aligned}
\]
Choosing $\lambda = \alpha t$ with $t > 0$ we can write
\[
\Pi\left(\sum_{j > N}w_j >R \right) \leq e^{-\alpha t \log (R) - N\log (1+t)}.
\]
Fix now $R = n^{-\beta}$ and $N = \ceil{\beta' \log(n)}$, so that we have
\[
\Pi\left(\sum_{j > \ceil{\beta' \log n}}w_j > \frac{1}{n^\beta} \right) \leq e^{\log(n) \left[ \alpha\beta t  - \beta'\log (1+t)\right]}.
\]
Fix now $\beta' \geq e\alpha \beta$ and $t = \frac{\beta'}{\alpha\beta}-1$, so that
\[
\Pi\left(\sum_{j > \ceil{\beta' \log n}}W_j > \frac{1}{n^\beta} \right) \leq n^{-\alpha \beta},
\]
as desired.

As regards point $2$, by similar calculations we get
\[
\Pi\left(\sum_{j > \ceil{\beta' \log n}}w_j <  \frac{1}{n^\beta} \right) \leq e^{\log(n) \left[-\alpha \beta t - \beta' \log(1-t) \right]},
\]
with $t \in (0, 1)$. Consider now $\beta' < \alpha \beta$ and choose $t = 1-\frac{\beta'}{\alpha\beta}$, so that
\[
\Pi\left(\sum_{j > \ceil{\beta' \log n}}W_j <  \frac{1}{n^\beta} \right) \leq e^{\alpha \beta \log(n) \left[\frac{\beta'}{\alpha\beta} - \frac{\beta'}{\alpha\beta} \log \left(\frac{\beta'}{\alpha\beta} \right) - 1 \right]}.
\]
Since $h(x) = x - x \log(x) \to 0$ as $x \to 0$, we deduce that there exists $\beta^* = \beta^*(\delta) \leq \alpha\beta$ so that 
\[
\Pi\left(\sum_{j > \ceil{\beta' \log n}}W_j <  \frac{1}{n^\beta} \right) \leq n^{-(1-\delta)\alpha\beta},
\]
for every $\beta' \leq \beta^*$.
\end{proof}

\subsection{Proof of Lemma \ref{lemma: Kn_priori}}
\begin{proof}[Proof of Lemma \ref{lemma: Kn_priori}]
It is well-known \citep[Lemma 2.1]{korwar1973contributions} that
\[
\Phi \left(c_{k+1} \text{ is distinct from } c_j \text{ for every } j = 1, \dots, k \right) = \frac{\alpha}{\alpha+k},
\]
for every $k \geq 0$. Then notice that
\[
\begin{aligned}
\E\left[K_n \right] &= \alpha\sum_{k = 0}^{n-1}\frac{1}{\alpha+k} \leq 1+\alpha\int_0^{n-1}\frac{1}{\alpha+x}\, \d x =1+\alpha \log \left(1+\frac{n-1}{\alpha} \right),
\end{aligned}
\]
and
\[
\begin{aligned}
\E\left[K_n \right] &= \alpha\sum_{k = 0}^{n-1}\frac{1}{\alpha+k} \geq  \alpha\int_0^{n-1}\frac{1}{\alpha+x}\, \d x = \alpha \log \left(1+\frac{n-1}{\alpha} \right).
\end{aligned}
\]
As regards point $1$, for n large enough
\[
(1+\delta)\alpha \log \left(1+\frac{n-1}{\alpha} \right) > 3\E[K_n],
\]
and therefore the Chernoff bounds imply that
\[
\begin{aligned}
\Phi\biggl(K_n \geq (1+\delta)\alpha\log\left(1+\frac{n-1}{\alpha} \right) \biggr)& \leq  \P\biggl(K_n \geq 3\E[K_n] \biggr) \leq e^{-\alpha \log\left( 1+\frac{n-1}{\alpha}\right)},
\end{aligned}
\]
which means
\[
\Phi\biggl(K_n \geq (1+\delta)\alpha\log\left(1+\frac{n-1}{\alpha} \right) \biggr) \leq \left(1+\frac{n-1}{\alpha} \right)^{-\alpha} \leq cn^{-\alpha},
\]
 for some $c = c(\alpha)$ as desired. As regards point $2$, again by the Chernoff bounds we obtain
\[
\begin{aligned}
\Phi\biggl(K_n \leq (1-\delta)\alpha\log\left(1+\frac{n}{\alpha} \right) \biggr) &\leq \left(\frac{e^{-\delta}}{(1-\delta)^{1-\delta}} \right)^{\alpha \log\left( 1+\frac{n-1}{\alpha}\right)}\\
& = \left(1+\frac{n-1}{\alpha} \right)^{-\alpha\left[\delta - (1-\delta)\log(1-\delta) \right]}.
\end{aligned}
\]
Since $h(\delta) = \delta - (1-\delta)\log(1-\delta) \to 1$ as $\delta \to 1$, there exists $\delta \in (0,1)$ such that $h(\delta) = 3/4$, as desired.
\end{proof}

\subsection{Proof of Theorem \ref{thm:stick_post}}
Denote with $||f-g||_1$ the $L^1$ distance between probability density functions $f$ and $g$. We first prove a preliminary lemma on the rate of convergence for densities under the Hellinger distance: this follows naturally from the classical theory pioneered in \cite{ghosal2000convergence} and developed in subsequent works \citep{ghosal2001entropies, Ghosal2007, Subhashis2017}. See Theorem $4.3$ in \cite{hairault2022evidence} for a similar result in $L^1$ distance.
\begin{lemma}\label{lm:post_consistency}
Under assumptions $(A1)-(A3)$ and $(B1)-(B3)$, there exists $q > 0$ such that
\[
\Pi\left(||Pf - P^*f||_1 \geq \frac{(\log n)^q}{\sqrt{n}} \mid X_{1:n}\right) \to 0,
\]
as $n \to \infty$ in $Q^{(\infty)}$-probability.
\end{lemma}
\begin{proof}
This is a special case of Theorem $4.3$ in \cite{hairault2022evidence}, whose assumptions are implied by $(A1)-(A3)$ and $(B1)-(B3)$.
\end{proof}
We can now give the proof of Theorem \ref{thm:stick_post}. 
\begin{proof}[Proof of Theorem \ref{thm:stick_post}]
Fix $\delta > 0$ and $\epsilon > 0$. Define the event
\[
A_n = \left\{X_{1:n} \mid \int \prod_{i = 1}^n\frac{P f(X_i)}{P^* f(X_i)}\Pi(\d P) \geq cn^{-DK/2 - (K-1)/2 -\alpha/2} \right\}.
\]
By Theorem \ref{thm:evidence_lower_bound} there exists $c = c(\epsilon) > 0$ such that $Q^{(n)}(A_n) \geq 1-\epsilon/2$. 

Combined with Lemma \ref{lm:post_consistency}, it thus suffices to prove that
\[
\begin{aligned}
\E\biggl[&\mathbbm{1}_{A_n}(X_{1:n})\Pi\left(\left\{\sum_{j > K}w_j > \frac{1}{n^{1/2-\delta}}\right\} \cap \left\{||Pf - P^*f||_1 \leq \frac{(\log n)^q}{\sqrt{n}} \right\} \mid X_{1:n} \right) \biggr]\\
& \leq cn^{DK/2 + (K-1)/2 +\alpha/2}\Pi\left(\left\{\sum_{j > K}w_j > \frac{1}{n^{1/2-\delta}}\right\} \cap \left\{||Pf - P^*f||_1 \leq \frac{(\log n)^q}{\sqrt{n}} \right\} \right) \to 0,
\end{aligned}
\]
as $n \to \infty$. This is implied by Proposition \ref{prop:prior_bound} below.
\end{proof}
\begin{proposition}\label{prop:prior_bound}
Consider the same setting of Theorem \ref{thm:stick_post}. Then for every $\alpha > 0$ and $\delta > 0$ there exists $\beta > 0$ such that
\[
\Pi\left(\left\{\sum_{j > K}w_j > \frac{1}{n^{1/2-\delta}}\right\} \cap \left\{||Pf - P^*f||_1 \leq \frac{(\log n)^q}{\sqrt{n}} \right\} \right) < n^{-DK/2 - (K-1)/2 -\alpha/2 -\beta},
\]
for $n$ greater than some fixed $N$.
\end{proposition}
The proof of Proposition \ref{prop:prior_bound} requires multiple steps, which are given in the next subsection.

\subsubsection{Proof of Proposition \ref{prop:prior_bound}}
We first state in our context a known result from \cite{gassiat2014local}.
\begin{lemma}\label{lm:lower_bound_consistency}
Let $P^*$ as in $(A2)$. Then under $(B1)-(B3)$ there exists $\epsilon^* = \epsilon^*(Q) > 0$ such that, for every $\epsilon^* > \epsilon > 0$, there exists $m = m(Q, \epsilon) > 0$ such that for every $P = \sum_{j \geq 1}w_j \delta_{\theta_j}$ it holds that
\[
\begin{aligned}
||Pf - P^*f||_1 \geq m&\biggl\{ \sum_{j \, :\, \theta_j \in B_0} w_j + \sum_{k = 1}^K\left\lvert\sum_{j \, :\, \theta_j \in B_k} w_j - w_k ^* \right\rvert \\
&+ \sum_{k = 1}^K\left\lvert\left\lvert \sum_{j \, :\, \theta_j \in B_k} w_j(\theta_j-\theta_k^*) \right\rvert\right\rvert + \frac{1}{2}\sum_{k = 1}^K \sum_{j \, :\, \theta_j \in B_k} w_j\left\lvert\left\lvert\theta_j-\theta_k^* \right\rvert\right\rvert^2 \biggr\},
\end{aligned}
\]
where $B_k = B_\epsilon(\theta_k^*)$ is the ball of radius $\epsilon$ and center $\theta_k^*$ and $B_0 = \left(\cup_{k = 1}^KB_k\right)^c$.
\end{lemma}
\begin{proof}
This follows by Theorem $3.10$ in \cite{gassiat2014local}, whose conditions $(1)-(2)$ are easily shown to be satisfied by $(B1)-(B3)$. The fact that $\epsilon$ can be chosen arbitrarily small comes from Lemma $3.8$ therein.
\end{proof}
Fix $\delta > 0$ and consider $\epsilon > 0$ such that $\epsilon^* > \epsilon > 0$ with $\epsilon^*$ as in Lemma \ref{lm:lower_bound_consistency}. For every $\epsilon^*$ we denote $B_k$ as in Lemma \ref{lm:lower_bound_consistency} and $\alpha_k = \alpha(B_k)$, with $k = 0, \dots, K$. We will fix $\epsilon$ later in the proof: at the moment we choose it small enough so that $\alpha_k < 1$ for $k = 1, \dots, K$. 

Let $P \overset{\d}{=} \sum_{j \geq 1}w_j\delta_{\theta_j}$ be as in \eqref{GEM}. We introduce some notation that will be heavily used in the following. 
\begin{equation}\label{eq:notation1}
P_k = \sum_{j = 1}^Kw_j\delta_{\theta_j} \mathbbm{1}_{\left\{\theta_j \in B_k\right\}}, \quad p_k = P_k(\R^D), \quad R_K = \sum_{j > K}w_j
\end{equation}
be quantities related to the first $K$ terms of the mixtures, and let
\begin{equation}\label{eq:notation2}
\tP_k = \sum_{j \geq 1}\tw_j\delta_{\ttheta_j} \mathbbm{1}_{\left\{\ttheta_j \in B_k\right\}}, \quad \tp_k = \tP_k(\R^D), \quad \tQ_k = \frac{\tP_k}{\tp_k}
\end{equation}
be quantities related to the remaining terms, where
\[
\tw_k = \frac{w_{K+k}}{R_K}, \quad \ttheta_k = \theta_{K+k}.
\]
Moreover, define the following functionals of $P_k$ and $\tP_k$
\begin{equation}\label{eq:notation3}
F_k = \int (\theta-\theta_k^*)P_k(\d \theta), \quad \tF_k = \int (\theta-\theta_k^*)\tQ_k(\d \theta).
\end{equation}
We list some basic properties of the above quantities in the next two lemmas.
\begin{lemma}\label{lm:basic_properties1}
Consider the expressions in \eqref{eq:notation1}, \eqref{eq:notation2} and \eqref{eq:notation3}. Then it holds that
\begin{equation}\label{eq:basic_equality1}
\sum_{j \, :\, \theta_j \in B_0} w_j = p_0+R_K\tp_0, \quad \left\lvert\sum_{j \, :\, \theta_j \in B_k} w_j - w_k ^* \right\rvert = \left\lvert p_k + R_K\tp_k - w_k ^* \right\rvert 
\end{equation}
and
\begin{equation}\label{eq:basic_equality2}
\sum_{j \, :\, \theta_j \in B_k} w_j(\theta_j-\theta_k^*)  = F_k+R_K\tp_k \tF_k, 
\end{equation}
and
\begin{equation}\label{eq:basic_equality3}
    \sum_{j \, :\, \theta_j \in B_k} w_j\left\lvert\left\lvert\theta_j-\theta_k^* \right\rvert\right\rvert^2 = \int \lnorm \theta-\theta_k^* \rnorm^2P_k(\d \theta) + R_K\tp_k\int \lnorm \theta-\theta_k^* \rnorm^2\tQ_k(\d \theta),
\end{equation}
for every $k = 1, \dots, K$. Moreover
\begin{enumerate}
    \item $P_k$ and $\tP_{k'}$ are independent, for every $k,k' = 0, \dots, K$.

    \item $\tp_k$ and $\tQ_k$ are independent for every $k = 1, \dots, K$.

    \item $\tQ_k \sim \text{DP}\left(\alpha P_0(\cdot \cap B_k)/P_0(B_k) \right)$ for $k = 0, \dots, K$ and
    \[
    \left(\tp_0, \tp_1, \dots, \tp_K \right) \sim \text{Dirichlet}(\alpha_0, \alpha_1, \dots, \alpha_K)
    \]
\end{enumerate}
\end{lemma}
\begin{proof}
The equalities in \eqref{eq:basic_equality1}, \eqref{eq:basic_equality2} and \eqref{eq:basic_equality3} follow immediately from the definitions in \eqref{eq:notation1}, \eqref{eq:notation2} and \eqref{eq:notation3}.

As regards point $1$, notice that
\[
w_j = v_j\prod_{i = 1}^{j-1}(1-v_i), \quad  \tw_j = v_{K + j}\prod_{i = K+1}^{K+j-1}(1-v_i).
\]
for every $j \geq 1$. Since $\{v_j\}_j$ is a sequence of independent random variables, the independence of $P_k$ and $\tP_{k'}$ follows. For the same reason, by \eqref{GEM} we deduce that
\[
\tP = \sum_{j \geq 1} \tw_j\delta_{\ttheta_j} \sim \text{DP}(\alpha P_0).
\]
Then points 2 and 3 follow from the fact that the Dirichlet process is tail-free (see e.g.\ Remark $3.2.1$ of \cite{ghosh2003bayesian}).
\end{proof}
\begin{lemma}\label{lm:basic_properties2}
    Assume that
    \[
    \int \lnorm \theta-\theta_k^* \rnorm^2P_k(\d \theta)  \leq r_n \quad \text{and} \quad R_K\tp_k\int \lnorm \theta-\theta_k^* \rnorm^2\tQ_k(\d \theta)  \leq r_n
    \]
    for $k = 1, \dots, K$. Then we have that
    \[
    \lnorm F_k \rnorm \leq \sqrt{p_kr_n} \quad \text{and} \quad R_K\tp_k\lnorm \tF_k \rnorm \leq \sqrt{R_K \tp_k r_n}.
    \]
\end{lemma}
\begin{proof}
By definition of $F_k$ and $\tF_k$, by Jensen's inequality we have that
\[
\left(\frac{\lnorm F_k \rnorm}{p_k}\right)^2 \leq \frac{\int \lnorm \theta-\theta_k^* \rnorm^2P_k(\d \theta)}{p_k}
\]
and
\[
\lnorm \tF_k \rnorm^2 \leq \int \lnorm \theta-\theta_k^* \rnorm^2\tQ_k(\d \theta),
\]
from which the results immediately follow.
\end{proof}
Thus, setting $r_n = m^{-1}(\log n)^q/\sqrt{n}$, by Lemma \ref{lm:lower_bound_consistency} combined with Lemmas \ref{lm:basic_properties1} and \ref{lm:basic_properties2}, in order to prove the result it suffices to show that
\begin{equation}\label{eq:new_statement}
\Pi\left(B_n \right) < n^{-DK/2 - (K-1)/2 -\alpha/2 -\beta},
\end{equation}
where
\begin{equation}\label{eq:trans_bound1}
\begin{aligned}
    B_n = & \left\{R_K > \frac{1}{n^{1/2-\delta}} \right\} \cap \left\{p_0 \leq r_n \right\}\cap \left\{R_K\tp_0 \leq r_n \right\}\\
    &\quad \cap \bigcap_{k = 1}^K \biggl\{\left\{\left\lvert p_k + R_K\tp_k - w_k ^* \right\rvert  \leq r_n\right \} \cap \left\{\lnorm F_k \rnorm \leq \sqrt{p_kr_n}\right \} \\
    & \qquad \qquad  \cap  \left\{\lnorm F_k+R_K\tp_k \tF_k \rnorm  \leq r_n\right \} \cap \left\{R_K\tp_k\lnorm \tF_k \rnorm \leq \sqrt{R_K \tp_k r_n}\right \} \biggr\}.
\end{aligned}
\end{equation}

The rest of the proof is divided in several steps.

\paragraph{First step.} Fix $R_K$, $P_k$, $F_k$ and $\tp_k$, with $k = 1, \dots, K$. We want to upper bound the probability of the event
\begin{equation}\label{eq:to_bound1}
\bigcap_{k = 1}^K \left\{\left\{\lnorm F_k+R_K\tp_k \tF_k \rnorm  \leq r_n\right \} \cap \left\{R_K\tp_k\lnorm \tF_k \rnorm \leq \sqrt{R_K \tp_k r_n}\right \} \right\},
\end{equation}
where $\tF_1, \dots, \tF_K$ are the only random objects.

First of all, notice that if $\lnorm F_k \rnorm > \sqrt{R_K\tp_kr_n} + r_n$ then we have that
\[
\lnorm F_k+R_K\tp_k \tF_k \rnorm \geq \lnorm F_k \rnorm - R_K\tp_k\lnorm  \tF_k \rnorm > r_n.
\]
Therefore we need $\lnorm F_k \rnorm \leq \sqrt{R_K\tp_kr_n} + r_n$.  Moreover we have that
\[
\Pi\left(\lnorm F_k+R_K\tp_k \tF_k \rnorm  \leq r_n \right) = \Pi\left(\tF_k \in B_{\frac{r_n}{R_K\tp_k}}\left(\frac{F_k}{R_K\tp_k} \right) \right)
\]
and since, by Lemma $4.8$ in \cite{hairault2022evidence}, $\tF_k$ admits a bounded density with respect to the Lebesgue measure on $\R^D$, we have that
\[
\Pi\left(\lnorm F_k+R_K\tp_k \tF_k \rnorm  \leq r_n \right) \leq 
\begin{cases}
    1 \quad \text{if } R_K\tp_k \lesssim r_n\\
    \\
    B\left(\frac{r_n}{R_K\tp_k} \right)^D\quad \text{else}
\end{cases}
\]
for some fixed constant $B$. By definition of $B_n$ we need that $\left\lvert p_k + R_K\tp_k - w_k ^* \right\rvert  \leq r_n$, which implies that
\[
\left\lvert p_k - w_k ^* \right\rvert  \leq R_K\tp_k + r_n.
\]
Therefore we conclude that
\begin{equation}\label{eq:bound_atoms}
\Pi\left(\lnorm F_k+R_K\tp_k \tF_k \rnorm  \leq r_n \right) \leq 
\begin{cases}
    1 \quad \text{if } \left\lvert p_k - w_k ^* \right\rvert  \lesssim r_n\\
    \\
    B\left(\frac{r_n}{R_K\tp_k} \right)^D\quad \text{else}
\end{cases}
\end{equation}
Define the sets
\[
\matK_1 = \left\{k \in \{1, \dots, K\} \, \mid \, \left\lvert p_k - w_k ^* \right\rvert  \leq 2r_n \right\}
\]
and
\[
\matK_2 = \left\{k \in \{1, \dots, K\} \, \mid \, \left\lvert p_k - w_k ^* \right\rvert  > 2r_n \right\},
\]
with respective cardinalities $K_1$ and $K_2$: notice that, since $R_K > n^{-1/2+\delta}$ in $B_n$, we deduce that $\matK_2$ is non-empty. Moreover, if $k \in \matK_1$ we immediately get that
\[
\lnorm F_k \rnorm \leq 2Cr_n \quad \text{and} \quad R_K\tp_k \leq 3r_n, 
\]
where $C$ is the diameter of $\Theta$, which is finite by $(A2)$. If instead $k \in \matK_2$ it follows that $w_k^*-p_k > 2r_n$ and therefore $(w_k^*-p_k)/2 \leq R_K\tp_k \leq 2(w_k^*-p_k)$. Combining this with $\lnorm F_k \rnorm \leq \sqrt{R_K\tp_kr_n} + r_n$, we deduce that
\[
\lnorm F_k \rnorm \leq 2\sqrt{(w_k^*-p_k)r_n} \quad \text{and} \quad \tp_k \geq \frac{w_k^*-p_k}{2R_K}
\]
and therefore
\[
\prod_{k \in \matK_2}\tp_k^{-1} \leq \left(2R_K\right)^{K_2D}\prod_{k \in \matK_2}\left(w_k^*-p_k \right)^{-D}
\]
for every $k \in \matK_2$.
Combining the above with \eqref{eq:bound_atoms} we conclude that the probability of the event in \eqref{eq:to_bound1} can be upper bounded by
\begin{equation}\label{eq:final_bound1}
\left(2B\right)^{K_2D}r_n^{K_2D}\left(\prod_{k \in \matK_2}\left(w_k^*-p_k \right)^{-D}\mathbbm{1}_{\left\{ \lnorm F_k \rnorm \leq 2\sqrt{(w_k^*-p_k)r_n}\right\}}\right)\left(\prod_{k \in \matK_1}\mathbbm{1}_{\left\{ \lnorm F_k \rnorm \leq (2C+1)r_n\right\}}\right).
\end{equation}
Notice that \eqref{eq:final_bound1} does not depend on $\tp_0, \dots, \tp_K$.

\paragraph{Second step.} Conditional on $R_K$ and $p_k$, with $k = 0, \dots, K$, we want to upper bound the probability of the event
\begin{equation}\label{eq:to_bound2}
\left\{R_K\tp_0 \leq r_n \right\}\cap \bigcap_{k = 1}^K \left\{\left\lvert p_k + R_K\tp_k - w_k ^* \right\rvert  \leq r_n\right \},
\end{equation}
where $(\tp_0, \dots, \tp_K)$ is a random vector distributed as Dirichlet$(\alpha_0, \dots, \alpha_K)$ from point $3$ of Lemma \ref{lm:basic_properties1}. Given $\matK_1$ and $\matK_2$ as in the previous point, we need to upper bound
\[
\int_{\Delta_K} \left(t_0\mathbbm{1}_{\left\{ R_Kt_0 \leq r_n\right\}} \right)^{\alpha_0-1}\left(\prod_{k \in \matK_1}t_k^{\alpha_k-1}\mathbbm{1}_{\left\{ R_Kt_k \leq 3r_n\right\}}\right)\left(\prod_{k \in \matK_2}t_k^{\alpha_k-1}\mathbbm{1}_{\left\{ \left\lvert R_Kt_k+p_k - w_k ^* \right\rvert  \leq r_n\right\}}\right)\, \d \underline{t},
\]
where $\underline{t} = (t_1, \dots, t_K)$, $t_0 = 1-\sum_{k=1}^Kt_k$ and $\Delta_K$ is the $K$-dimensional simplex. Assume without loss of generality that $K \in \matK_2$ and, since $\alpha_K < 1$, we have that
\[
t_K^{\alpha_K-1} \leq \left( \frac{w_K^*-p_K-r_n}{R_K}\right)^{\alpha_K-1}\leq 2R_K\frac{(w_K^*-p_K)^{\alpha_K-1}}{R_K^{\alpha_K}},
\]
since $w_K^*-p_K > 2r_n$. Consider the change of variables
\[
s_0 = \sum_{k=1}^Kt_k, \quad s_k = t_k,
\]
with $k = 1, \dots, K-1$. Since the Jacobian of this transformation has unit determinant, the integral above is upper bounded by
\begin{equation}\label{eq:split_integral1}
\begin{aligned}
2R_K\frac{(w_K^*-p_K)^{\alpha_K-1}}{R_K^{\alpha_K}}\int_{\R^K} \bigl((1-s_0)^{\alpha_0-1}&\mathbbm{1}_{\left\{ 0 \leq R_K(1-s_0) \leq r_n\right\}} \bigr)\left(\prod_{k \in \matK_1}s_k^{\alpha_k-1}\mathbbm{1}_{\left\{0\leq R_Ks_k \leq 3r_n\right\}}\right)\\
&\times \left(\prod_{k \in \matK_2\backslash K}s_k^{\alpha_k-1}\mathbbm{1}_{\left\{0\leq \left\lvert R_Ks_k+p_k - w_k ^* \right\rvert  \leq r_n\right\}}\right)\, \d \underline{s},
\end{aligned}
\end{equation}
which can be factorized as
\[
\left(\int_{1-r_n/R_K}^1(1-s_0)^{\alpha_0-1}\, \d s_0 \right)\left(\prod_{k \in \matK_1}\int_0^{3r_n/R_K}s_k^{\alpha_k-1} \, \d s_k\right)\left(\prod_{k \in \matK_2\backslash K}\int_{(w_k^*-p_k-r_n)/R_K}^{(w_k^*-p_k+r_n)/R_K}s_k^{\alpha_k-1} \, \d s_k\right).
\]
With easy calculations we obtain that
\begin{equation}\label{eq:split_integral2}
\int_{1-r_n/R_K}^1(1-s_0)^{\alpha_0-1}\, \d s_0 = \frac{1}{\alpha_0}\left(\frac{r_n}{R_K} \right)^{\alpha_0}
\end{equation}
and
\begin{equation}\label{eq:split_integral3}
\begin{aligned}
\int_{(w_k^*-p_k-r_n)/R_K}^{(w_k^*-p_k+r_n)/R_K}s_k^{\alpha_k-1} \, \d s_k &= R_K^{-\alpha_k}\int_{w_k^*-p_k-r_n}^{w_k^*-p_k+r_n}x_k^{\alpha_k-1} \, \d x_k =  \left(\frac{w_k^*-p_k}{R_K} \right)^{\alpha_k}\int_{1-r_n/(w_k^*-p_k)}^{1+r_n/(w_k^*-p_k)}y_k^{\alpha_k-1} \, \d y_k\\
&\leq 2\left(\frac{w_k^*-p_k}{R_K} \right)^{\alpha_k}\left( 1-\frac{r_n}{w_k^*-p_k}\right)^{\alpha_k-1}\frac{r_n}{w_k^*-p_k} \leq \frac{(w_k^*-p_k)^{\alpha_k-1}}{R_K^{\alpha_k}}r_n,
\end{aligned}
\end{equation}
for $k \in \matK_2$ and
\begin{equation}\label{eq:split_integral4}
\int_0^{3r_n/R_K}s_k^{\alpha_k-1} \, \d s_k = \frac{3^{\alpha_k}}{\alpha_k}\left(\frac{r_n}{R_K} \right)^{\alpha_k} = \frac{3^{\alpha_k}}{\alpha_k}\frac{r_n^{\alpha_k-1}}{R_K^{\alpha_k}}r_n \leq 2\frac{3^{\alpha_k}}{\alpha_k}\frac{|w_k^*-p_k|^{\alpha_k-1}}{R_K^{\alpha_k}}r_n,
\end{equation}
for $k \in \matK_1$. Combining \eqref{eq:split_integral1}, \eqref{eq:split_integral2}, \eqref{eq:split_integral3} and \eqref{eq:split_integral4} we can conclude that the probability of the event in \eqref{eq:to_bound2} is upper bounded by
\[
Mr_n^{K+\alpha_0-1}R_K^{1-\alpha}\left(\prod_{k \in \matK_2}(w_k^*-p_k)^{\alpha_k-1}\mathbbm{1}_{\left\{p_k \leq w_k^*-2r_n \right\} }\right)\left(\prod_{k \in \matK_1}|w_k^*-p_k|^{\alpha_k-1}\mathbbm{1}_{\left\{|w_k^*-p_k| \leq 2r_n \right\} } \right).
\]
Define two further subsets of $\matK_2$
\[
\matK_{21} = \left\{k \in \matK_2 \, \mid \, p_k \leq r_n \right\}, \quad \matK_{22} = \left\{k \in \matK_2 \, \mid \, p_k > r_n \right\},
\]
depending on the size of $p_k$, and combining the above with \eqref{eq:trans_bound1} and \eqref{eq:final_bound1} we conclude that
\begin{equation}\label{eq:conditional_bound}
\begin{aligned}
\Pi\left( B_n \mid w_{1:K}, \theta_{1:K}  \right) \leq M'&r_n^{K_2D+K+\alpha_0-1}R_K^{1-\alpha}\mathbbm{1}_{\left\{R_K > \frac{1}{n^{1/2-\delta}} \right\}}\mathbbm{1}_{\left\{p_0 \leq r_n \right\}}\\
&\times\left(\prod_{k \in \matK_{22}}(w_k^*-p_k)^{\alpha_k-D-1} \mathbbm{1}_{\left\{r_n \leq p_k \leq w_k^*-2r_n \right\}}\mathbbm{1}_{\left\{ \lnorm F_k \rnorm \leq 2\sqrt{\min\{p_k,w_k^*-p_k\}r_n}\right\}}\right)\\
&\times\left(\prod_{k \in \matK_{21}}(w_k^*-p_k)^{\alpha_k-1} \mathbbm{1}_{\left\{p_k \leq r_n \right\}}\right)\\
&\times\left(\prod_{k \in \matK_1}|w_k^*-p_k|^{\alpha_k-1} \mathbbm{1}_{\left\{|w_k^*-p_k| \leq 2r_n \right\} }\mathbbm{1}_{\left\{ \lnorm F_k \rnorm \leq (2C+1)r_n\right\}}\right),
\end{aligned}
\end{equation}
for another absolute constant $M'$. We are now left with marginalizing \eqref{eq:conditional_bound} with respect to $w_{1:K}$ and $\theta_{1:K}$.

\paragraph{Third step.} We first need a preliminary lemma.
\begin{lemma}\label{lm:convolution}
Let $\theta_i \simiid G$, where $i = 1, \dots, I$ and $G$ is probability distribution on $\R^D$ with bounded density $g$ with respect to the Lebesgue measure. Let $g_{\bm{w}}$ be the density of $\sum_{i = 1}^Iw_i\theta_i$, with $\bm{w} \in \Delta_{I-1}$. Then
\[
\sup_{\bm{w} \in \Delta_{I-1}} \, \sup_{\theta \in \R^D} \, g_{\bm{w}}(\theta) < \infty.
\]
\end{lemma}
\begin{proof}
Fix $\bm{w} \in \Delta_{I-1}$ and without loss of generality assume that $w_1 \geq 1/I$. Then
\[
\sum_{i = 1}^Iw_i\theta_i = w_1\theta_1+\sum_{i = 2}^Iw_i\theta_i =: X+Y,
\]
where $X$ and $Y$ are independent random vectors with densities $h_X$ and $h_Y$. In particular
\[
h_X(x) = w_1^{-D}g\left(\frac{x}{w_1} \right) \quad \text{and} \quad g_{\bm{w}}(\theta) = \int_{\R^D}h_X(\theta-y)h_Y(y) \, \d y.
\]
Therefore
\[
\sup_{\theta \in \R^D} \,g_{\bm{w}}(\theta) \leq I^D\sup_{\theta \in \R^D} \,g(\theta),
\]
which is finite and independent from $\bm{w}$.
\end{proof}
Let $I_k = \{ i = 1:K \, \mid \, \theta_i \in B_k\}$, with $k = 0, \dots, K$. Then
\[
F_k = \sum_{i \in I_k}w_i(\theta_i-\theta_k^*) \quad \text{and} \quad p_k = \sum_{i \in I_k}w_i.
\]
We can now decompose the indicator functions in \eqref{eq:conditional_bound} according to all the possible choices of $I_0, \dots, I_K$. Since there are finitely many such combinations, it suffices to show the result individually for each of them. Then for the rest of the proof we fix $I_0, \dots, I_K$.

Notice that $F_k/p_k$ satisfies the hypotheses of Lemma \ref{lm:convolution} and therefore
\[
\Pi\left(\lnorm F_k \rnorm \leq (2C+1)r_n \right) \leq Rr_n^D
\]
for $k \in \matK_1$ and
\[
\Pi \left(\lnorm F_k \rnorm \leq 2\sqrt{\min\{p_k,w_k^*-p_k\}r_n} \right) \leq R\left(p_k^{-1}\sqrt{\min\{p_k,w_k^*-p_k\}r_n} \right)^D,
\]
for $k \in \matK_{22}$, where $R$ is a fixed constant. Combining this with \eqref{eq:conditional_bound}, we have that
\begin{equation}\label{eq:conditional_bound2}
\begin{aligned}
\Pi\left( B_n \mid w_{1:K}, I_0, \dots, I_K  \right) \leq R'&r_n^{KD+K+\alpha_0-1}R_K^{1-\alpha}\mathbbm{1}_{\left\{R_K > \frac{1}{n^{1/2-\delta}} \right\}}\mathbbm{1}_{\left\{p_0 \leq r_n \right\}}\\
&\times\left(\prod_{k \in \matK_{22}}(w_k^*-p_k)^{\alpha_k-D-1}\left(p_k^{-1}\sqrt{\min\{p_k,w_k^*-p_k\}r_n} \right)^D \mathbbm{1}_{\left\{r_n \leq p_k \leq w_k^*-2r_n \right\}}\right)\\
&\times\left(\prod_{k \in \matK_{21}}(w_k^*-p_k)^{\alpha_k-1} \mathbbm{1}_{\left\{p_k \leq r_n \right\}}\right)\left(\prod_{k \in \matK_1}|w_k^*-p_k|^{\alpha_k-1} \mathbbm{1}_{\left\{|w_k^*-p_k| \leq 2r_n \right\} }\right),
\end{aligned}
\end{equation}
for a fixed constant $R'$.

\paragraph{Fourth step.} It is known (see e.g.\ \cite{connor1969concepts} and Section $3.1$ in \cite{ishwaran2001gibbs} that $w_{1:K}$ has a generalized Dirichlet distribution with density
\begin{equation}\label{eq:generalized_dirichlet}
h(w_1, \dots, w_k) = \alpha^KR_K^{\alpha-1}\prod_{k = 1}^{K-1}\left(1-\sum_{j=1}^kw_j \right)^{-1}\mathbbm{1}_{\{w_{1:K} \in \Delta_{K-1}\}}.
\end{equation}
If $R_K \geq (\min_k w_k^*)/2$, then $\prod_{k = 1}^{K-1}\left(1-\sum_{j=1}^kw_j \right)^{-1}$ is trivially bounded above. If instead $R_K < (\min_k w_k^*)/2$, then $p_k \geq (\min_k w_k^*)/2 - r_n$ for every $k = 1,\dots, K$. Indeed otherwise
\[
|p_k-w_k^*+R_K\tp_k| \geq w_k^*-p_k-R_K > r_n,
\]
which is incompatible with the definition of $B_n$. This implies that $I_k = \{k'\}$ for every $k = 1, \dots, K$, i.e.\, each atom $\theta_{k'}$ is associated to a unique set $B_k$. In this case we conclude that $w_k \geq (\min_k w_k^*)/4$ and therefore $\prod_{k = 1}^{K-1}\left(1-\sum_{j=1}^kw_j \right)^{-1}$ is bounded. From \eqref{eq:conditional_bound2} then we obtain
\begin{equation}\label{eq:conditional_bound3}
\begin{aligned}
\Pi\left( B_n \mid  I_0, \dots, I_K  \right) \leq T&r_n^{KD+K+\alpha_0-1}\int_{\Delta_{K}}\mathbbm{1}_{\left\{R_K > \frac{1}{n^{1/2-\delta}} \right\}}\mathbbm{1}_{\left\{p_0 \leq r_n \right\}}\\
&\times\left(\prod_{k \in \matK_{22}}(w_k^*-p_k)^{\alpha_k-D-1}\left(p_k^{-1}\sqrt{\min\{p_k,w_k^*-p_k\}r_n} \right)^D \mathbbm{1}_{\left\{r_n \leq p_k \leq w_k^*-2r_n \right\}}\right)\\
&\times\left(\prod_{k \in \matK_{21}}(w_k^*-p_k)^{\alpha_k-1} \mathbbm{1}_{\left\{p_k \leq r_n \right\}}\right)\left(\prod_{k \in \matK_1}|w_k^*-p_k|^{\alpha_k-1} \mathbbm{1}_{\left\{|w_k^*-p_k| \leq 2r_n \right\} }\right) \, \d \underline{p},
\end{aligned}
\end{equation}
for a fixed constant $T$ and with $\underline{p} = (p_0, \dots, p_K)$. Notice that the term $R_K^{1-\alpha}$ in \eqref{eq:conditional_bound2} cancels out thanks to the matching term $R_K^{\alpha-1}$ in \eqref{eq:generalized_dirichlet}.

\paragraph{Fifth step.} Since $R_K > n^{-1/2+\delta}$, there exists $k \in \matK_{22}$ such that $w_k^*-p_k > n^{-\beta}$ with $\beta < 1/2$. Without loss of generality, say $k = 1$. Then we can upper bound \eqref{eq:conditional_bound3} with the following factorization
\[
\begin{aligned}
\Pi\left( B_n \mid  I_0, \dots, I_K  \right) \leq T&r_n^{KD+K+\alpha_0-1}\int_0^1\mathbbm{1}_{\left\{p_0 \leq r_n \right\}} \, \d p_0\\
&\times \int_0^1(w_1^*-p_1)^{\alpha_1-D-1}\left(p_1^{-1}\sqrt{\min\{p_1,w_1^*-p_1\}r_n} \right)^D \mathbbm{1}_{\left\{r_n \leq p_1 \leq w_1^*-n^{-\beta} \right\}}\, \d p_1\\
&\times\prod_{k \in \matK_{22}\backslash 1}\left(\int_0^1(w_k^*-p_k)^{\alpha_k-D-1}\left(p_k^{-1}\sqrt{\min\{p_k,w_k^*-p_k\}r_n} \right)^D \mathbbm{1}_{\left\{r_n \leq p_k \leq w_k^*-2r_n \right\}}\, \d p_k\right)\\
&\times\prod_{k \in \matK_{21}}\left(\int_0^1(w_k^*-p_k)^{\alpha_k-1} \mathbbm{1}_{\left\{p_k \leq r_n \right\}}\, \d p_k\right)\\
&\times\prod_{k \in \matK_1}\left(\int_0^1|w_k^*-p_k|^{\alpha_k-1} \mathbbm{1}_{\left\{|w_k^*-p_k| \leq 2r_n \right\} }\, \d p_k\right).
\end{aligned}
\]
It is immediate to show that if $k \in \matK_{21}$ and $k' \in \matK_{1}$ then
\[
\int_0^1(w_k^*-p_k)^{\alpha_k-1} \mathbbm{1}_{\left\{p_k \leq r_n \right\}}\, \d p_k \leq Wr_n^{\alpha_k}, \quad \int_0^1|w_{k'}^*-p_{k'}|^{\alpha_{k'}-1} \mathbbm{1}_{\left\{|w_{k'}^*-p_{k'}| \leq 2r_n \right\} }\, \d p_{k'} \leq Wr_n^{\alpha_{k'}},
\]
for some universal constant $W$. Instead, if $k \in \matK_{22}\backslash1$ notice that
\[
\begin{aligned}
\frac{\sqrt{\min\{p_k,w_k^*-p_k\}r_n}}{(w_k^*-p_k)p_k}\mathbbm{1}_{\left\{r_n \leq p_k \leq w_k^*-2r_n \right\}} \leq 1,
\end{aligned}
\]
and therefore
\[
\int_0^1(w_k^*-p_k)^{\alpha_k-D-1}\left(p_k^{-1}\sqrt{\min\{p_k,w_k^*-p_k\}r_n} \right)^D \mathbbm{1}_{\left\{r_n \leq p_k \leq w_k^*-2r_n \right\}}\, \d p_k \leq Wr_n^{\alpha_k}.
\]
Thus we are left with
\begin{equation}\label{eq:conditional_bound4}
\begin{aligned}
\Pi\left( B_n \mid  I_0, \dots, I_K  \right) \leq W'&r_n^{KD+K+\alpha-\alpha_1-1}\\
&\times \int_{r_n}^{w_1^*-n^{-\beta}}(w_1^*-p_1)^{\alpha_1-D-1}\left(p_1^{-1}\sqrt{\min\{p_1,w_1^*-p_1\}r_n} \right)^D \, \d p_1,
\end{aligned}
\end{equation}
for an absolute constant $W'$.

\paragraph{Sixth and final step.} We decompose the integral in \eqref{eq:conditional_bound4} in two parts. First of all
\[
\begin{aligned}
\int_{r_n}^{w_1^*/2}(w_1^*-p_1)^{\alpha_1-D-1}\left(p_1^{-1}\sqrt{\min\{p_1,w_1^*-p_1\}r_n} \right)^D \, \d p_1 &\leq Qr_n^{D/2}\int_{r_n}^{w_1^*/2}p_1^{-D/2} \, \d p_1\\
& \leq Q'r_n,
\end{aligned}
\]
for some constants $Q$ and $Q'$. Moreover
\[
\begin{aligned}
\int_{w_1^*/2}^{w_1^*-n^{-\beta}}(w_1^*-p_1)^{\alpha_1-D-1}&\left(p_1^{-1}\sqrt{\min\{p_1,w_1^*-p_1\}r_n} \right)^D \, \d p_1 \leq Qr_n^{D/2}\int_{w_1^*/2}^{w_1^*-n^{-\beta}}(w_1^*-p_1)^{\alpha_1-D/2-1} \, \d p_1\\
&\leq Q'r_n^{D/2}n^{\beta D/2-\alpha_1\beta} \leq Q'r_n^{D/2}n^{\beta D/2}.
\end{aligned}
\]
Combining this with \eqref{eq:conditional_bound4} we conclude that
\[
\Pi\left( B_n \mid  I_0, \dots, I_K  \right) \leq Lr_n^{KD+(K-1)+\alpha-\alpha_1}\,r_n^{D/2}n^{\beta D/2},
\]
for a fixed constant $L$, and finally by definition of $r_n$ we can write
\[
\Pi\left( B_n \mid  I_0, \dots, I_K  \right) \leq L'(\log n)^Fn^{-KD/2-(K-1)/2-(\alpha-\alpha_1)/2}\,n^{-D(1/2-\beta)/2},
\]
for some constants $L'$ and $F$. Since $\beta < 1/2$, choosing $\beta' = 1/2-\beta$, we conclude
\[
\Pi\left( B_n \mid  I_0, \dots, I_K  \right) \leq L'(\log n)^Fn^{-KD/2-(K-1)/2-\alpha/2}\,n^{(\alpha_1-D\beta')/2},
\]
and then \eqref{eq:new_statement} follows by choosing $\epsilon$ small enough so that $\alpha_1 < D\beta'$.

\subsection{Proof of Corollary \ref{crl:convergence_wasserstein}}

\begin{proof}
Fix $\delta \in (0, 1/2)$. By Lemma \ref{lm:post_consistency} and Theorem \ref{thm:stick_post} we have that
\[
\Pi\left(\left\{||Pf - P^*f||_1 < \frac{(\log n)^q}{\sqrt{n}}\right\} \cap \left\{\sum_{k > K}w_k \leq \frac{1}{n^{1/2-\delta/2}} \right\} \mid X_{1:n}\right) \to 1,
\]
as $n \to \infty$ in $Q^{(\infty)}$-probability. Therefore by Lemma \ref{lm:lower_bound_consistency} we can also deduce that
\begin{equation*}
\Pi \left(\exists \text{ $\sigma \in \mathcal{S}_K$ s.t. } \theta_{\sigma(k)} \in B_k \text{ for } k = 1, \dots, K \mid X_{1:n}\right) \to 1,
\end{equation*}
as $n \to \infty$ in $Q^{(\infty)}$-probability, where $B_k = B_\epsilon(\theta_k^*)$ with $\epsilon$ small enough so that $\{B_k\}_k$ is a collection of disjoint sets. Thus the set $\Omega = \Omega_1 \cap \Omega_2 \cap \Omega_3$ defined as
\[
\Omega_1 = \left\{(\w, \btheta) \, \mid \, ||Pf - P^*f||_1 < \frac{(\log n)^q}{\sqrt{n}}\right\}, \quad \Omega_2 = \left\{(\w, \btheta) \, \mid \, \sum_{k > K}w_k \leq \frac{1}{n^{1/2-\delta/2}}\right\}
\]
and
\[
\Omega_3 = \left\{(\w, \btheta) \, \mid \, \exists \text{ $\sigma \in \mathcal{S}_K$ s.t. } \theta_{\sigma(k)} \in B_k \text{ for } k = 1, \dots, K\right\}
\]
is such that $\Pi(\Omega \mid X_{1:n}) \to 1$ as $n \to \infty$ in $Q^{(\infty)}$-probability. Let $(\w, \btheta) \in \Omega$ and assume by contradiction that for the permutation $\sigma$ in $\Omega_3$ we have that
\[
\left \lvert w_k^*-w_{\sigma(k)} \right \rvert > \frac{1}{n^{1/2-\delta}},
\]
for some $k = 1, \dots, K$. Then we deduce that
\[
\left\lvert\sum_{j \, :\, \theta_j \in B_k} w_j - w_k ^* \right\rvert \geq \left \lvert w_k^*-w_{\sigma(k)} \right \rvert - \sum_{j > K}w_j \geq \frac{1}{n^{1/2-\delta/2}}
\]
for $n$ large enough by definition of $\Omega_2$. However by definition of $\Omega_1$ and Lemma \ref{lm:lower_bound_consistency} we also have that
\[
\left\lvert\sum_{j \, :\, \theta_j \in B_k} w_j - w_k ^* \right\rvert \leq \frac{1}{n^{1/2}}, 
\]
and therefore we find a contradiction. With an analogous argument we can  also show that
\[
\left\lvert\left \lvert \theta_k^*-\theta_{\sigma(k)} \right \rvert \right \rvert \leq \frac{1}{n^{1/2-\delta}},
\]
and therefore the first part of the corollary is proven.

As regards the second part, if $P = \sum_{k \geq 1}w_k\delta_{\theta_k}$ for every $\sigma \in \mathcal{S}_K$ denote
\[
m_k = \min \{w_k^*, w_{\sigma(k)}\}, \quad s_k^+ = \max \{ w_k^*-w_{\sigma(k)}, 0\}, \quad s_k^- = \max \{ w_{\sigma(k)}-w_k^*, 0\}
\]
and 
\[
S^+ = \sum_{k=1}^Ks_k^+, \quad S^- = \sum_{k = 1}^Ks_k^-,
\]
and notice that by construction
\[
m_k + s_k^+ = w_k^*, \quad m_k + s_k^- = w_{\sigma(k)}, \quad S^+ = S^-+\sum_{k > K}w_k.
\]
Then, we define $\gamma_\sigma \in \mathcal{P}(\Theta \times \Theta)$ as
\[
\gamma_\sigma(\d x_1, \d x_2) = \sum_{k = 1}^Km_k\delta_{\left(\theta_k^*, \theta_{\sigma(k)}\right)}(\d x_1, \d x_2) + \frac{1}{S^+}\left(\sum_{k = 1}^K s_k^+\delta_{\theta_k^*}(\d x_1)\right)\otimes \left(\sum_{k = 1}^K s_k^-\delta_{\theta_{\sigma(k)}}(\d x_2) + \sum_{k > K}w_k\delta_{\theta_k}(\d x_2) \right).
\]
It is easy to see that $\gamma_{\sigma} \in \mathcal{C}(P^*, P)$, and therefore
\[
\begin{aligned}
W_1(P^*, P) &\leq \sum_{k = 1}^Km_k\left\lvert\left \lvert \theta_k^*-\theta_{\sigma(k)} \right \rvert \right \rvert + \frac{1}{S^+}\sum_{k = 1}^K\sum_{k' = 1}^K s_k^+s_{k'}^-\left\lvert\left \lvert \theta_k^*-\theta_{k'} \right \rvert \right \rvert+\frac{1}{S^+}\sum_{k = 1}^K\sum_{k' > K} s_k^+w_{k'}\left\lvert\left \lvert \theta_k^*-\theta_{k'} \right \rvert \right \rvert\\
&\leq \sum_{k = 1}^Km_k\left\lvert\left \lvert \theta_k^*-\theta_{\sigma(k)} \right \rvert \right \rvert + RS^+,
\end{aligned}
\]
where $R = \text{diam}(\Theta)$, which is finite by assumption (A2). Then the result follows by the first part of the corollary.
\end{proof}

\subsection{Proof of Theorem \ref{thm:tail_stick_post}}
\begin{proof}
In order to prove the result it suffices to show that for every $\beta > 1/2$ it holds that
\begin{equation}\label{eq:first_theorem3}
\Pi\left(\sum_{j > \ceil{\bar{\beta} \log n}}w_j > \frac{1}{n^{\beta}} \mid X_{1:n} \right) \to 0
\end{equation}
and
\begin{equation}\label{eq:second_theorem3}
\Pi\left(\sum_{j > \ceil{\underline{\beta} \log n}}w_j < \frac{1}{n^{\beta}} \mid X_{1:n}\right) \to 0,
\end{equation}
as $n \to \infty$ in $Q^{(\infty)}$-probability. Fix $\epsilon > 0$ and define the event
\[
A_n = \left\{X_{1:n} \mid \int \prod_{i = 1}^n\frac{P f(X_i)}{P^* f(X_i)}\Pi(\d P) \geq cn^{-DK/2 - (K-1)/2 -\alpha/2} \right\}.
\]
By Theorem \ref{thm:evidence_lower_bound} there exists $c > 0$ such that $Q^{(n)}(A_n) \geq 1-\epsilon/2$. 

Therefore, as regards \eqref{eq:first_theorem3}, it suffices to prove that
\[
\begin{aligned}
\E\biggl[\mathbbm{1}_{A_n}(X_{1:n})&\Pi\left(\sum_{j > \ceil{\bar{\beta} \log n}}w_j > \frac{1}{n^{\beta}} \mid X_{1:n} \right) \biggr]\\
& \leq cn^{DK/2 + (K-1)/2 +\alpha/2}\Pi\left(\sum_{j > \ceil{\bar{\beta} \log n}}w_j > \frac{1}{n^{\beta}} \right) \to 0,
\end{aligned}
\]
as $n \to \infty$, with $\bar{\beta} > 0$. By Lemma \ref{lemma:weights_priori} (point $1$) if $\bar{\beta} \geq e\alpha\beta$ we have that
\[
\Pi\left(\sum_{j > \ceil{\bar{\beta} \log n}}w_j > \frac{1}{n^{\beta}} \right) \leq cn^{-\alpha\beta}.
\]
Thus \eqref{eq:first_theorem3} follows if $\alpha^* > \frac{DK+K-1}{2\beta-1}$.

As regards \eqref{eq:second_theorem3}, similarly we need to show
\[
\begin{aligned}
\E\biggl[\mathbbm{1}_{A_n}(X_{1:n})&\Pi\left(\sum_{j > \ceil{\underline{\beta} \log n}}w_j < \frac{1}{n^{\beta}} \mid X_{1:n} \right) \biggr]\\
& \leq cn^{DK/2 + (K-1)/2 +\alpha/2}\Pi\left(\sum_{j > \ceil{\underline{\beta} \log n}}w_j < \frac{1}{n^{\beta}} \right) \to 0,
\end{aligned}
\]
as $n \to \infty$, with $\underline{\beta} > 0$. By Lemma \ref{lemma:weights_priori} (point $2$), taking $r \in (0,1)$ such that $r < 1-\frac{1}{2\beta}$ there exists $\beta^* > 0$ such that for every $\underline{\beta} < \beta^*$ we have that
\[
\Pi\left(\sum_{j > \ceil{\underline{\beta} \log n}}w_j < \frac{1}{n^{\beta}} \right) \leq cn^{-\alpha(1-r)\beta}.
\]
Thus \eqref{eq:second_theorem3} follows if $\alpha^* > \frac{DK+K-1}{2(1-r)\beta-1}$.
\end{proof}

\subsection{Proof of Theorem \ref{thm:number_clusters_post}}
\begin{proof}
Fix $\epsilon > 0$ and define the event
\[
A_n = \left\{X_{1:n} \mid \int \prod_{i = 1}^n\frac{P f(X_i)}{P^* f(X_i)}\Pi(\d P) \geq cn^{-DK/2 - (K-1)/2 -\alpha/2} \right\}.
\]
Reasoning as in the proof of Theorem \ref{thm:tail_stick_post}, by Theorem \ref{thm:evidence_lower_bound} there exists $c > 0$ such that $Q^{(n)}(A_n) \geq 1-\epsilon$. Thus, it suffices to prove that
\[
\begin{aligned}
\E\biggl[\mathbbm{1}_{A_n}(X_{1:n})&\Phi\biggl(\left\{K_n > (1+\bar{\delta})\alpha\log(n)\right\}\cup \left\{K_n < (1-\underline{\delta})\alpha\log(n)\right\}\mid X_{1:n}\biggr) \biggr]\\
& \leq c^{-1}n^{DK/2 + (K-1)/2 +\alpha/2}\Phi\biggl(\left\{K_n > (1+\bar{\delta})\alpha\log(n)\right\}\cup \left\{K_n < (1-\underline{\delta})\alpha\log(n)\right\} \biggr) \to 0,
\end{aligned}
\]
as $n \to \infty$, with $\underline{\delta} \in (0,1)$ and $\bar{\delta} > 0$. By Lemma \ref{lemma: Kn_priori} the latter can be chosen so that for $c' = c'(\alpha) > 0$ it holds
\[
\Phi\biggl(\left\{K_n > (1+\bar{\delta})\alpha\log(n)\right\}\cup \left\{K_n < (1-\underline{\delta})\alpha\log(n)\right\} \biggr) \leq c'n^{-\frac{3}{4}\alpha}.
\]
The result follows by choosing $\alpha^* = 2DK+2(K-1)$, since $\epsilon$ is arbitrary.
\end{proof}

\subsection{Proof of Theorem \ref{thm:residual_clustering}}
\begin{proof}
We denote with $\P$ the probability distribution of $(P, c_{1:n})$ on $\mathcal{P}(\Theta) \times \N^n$ induced by model \eqref{DPM_model2}. Then by Theorem \ref{thm:stick_post} it suffices to prove that for every $\delta \in (0, 1/2)$ it holds that
\[
\P\biggl(\left\{ \frac{1}{n}\sum_{i = 1}^n\mathbbm{1}_{\{ c_i > K\}} > \frac{1}{n^{1/2-\delta}}\right\} \cap \left\{ \sum_{k > K}w_k < \frac{1}{n^{1/2-\delta/2}}  \right\} \mid X_{1:n}\biggr) \to 0,
\]
as $n \to \infty$ in $Q^{(\infty)}$-probability. Reasoning as in the proof of Theorem \ref{thm:tail_stick_post}, it is sufficient to prove that
\[
n^{DK/2 + (K-1)/2+\alpha/2}\P\biggl(\left\{ \frac{1}{n}\sum_{i = 1}^n\mathbbm{1}_{\{ c_i > K\}} > \frac{1}{n^{1/2-\delta}}\right\} \cap \left\{ \sum_{k > K}w_k < \frac{1}{n^{1/2-\delta/2}} \right\}\biggr) \to 0,
\]
as $n \to \infty$. Since
\[
\mathbbm{1}_{\{c_i > K\}} \mid \sum_{k > K}w_k \simiid \text{Bernoulli}\left( \sum_{k > K}w_k\right),
\]
by the Markov inequality we have that
\[
\begin{aligned}
\P&\biggl(\frac{1}{n}\sum_{i = 1}^n\mathbbm{1}_{\{ c_i > K\}} > \frac{1}{n^{1/2-\delta}} \mid  \sum_{k > K}w_k < \frac{1}{n^{1/2-\delta/2}}  \biggr)\leq e^{-n^{1/2+\delta}}\left(1+\frac{e}{n^{1/2-\delta/2}} \right)^n,
\end{aligned}
\]
which is asymptotically equivalent to
\[
e^{-n^{1/2+\delta}+en^{1/2+\delta/2}} \leq e^{-2\sqrt{n}}
\]
for $n$ large enough. Therefore the result follows.
\end{proof}

\subsection{Proof of Theorem \ref{thm:conv_clustering}}
We need two preliminary lemmas.
\begin{lemma}\label{lm:preliminary_cn}
Under assumptions $(A1)-(A3)$ and $(B1)-(B3)$, there exists $\beta' = \beta'(D, K) > 0$ such that
\[
\Phi\left(c_i \leq  \lceil \beta' \log n \rceil \text{ for every } i = 1, \dots, n \mid X_{1:n} \right) \to 1,
\]
as $n \to \infty$ in $Q^{(\infty)}$-probability.
\end{lemma}
\begin{proof}
Fix $\beta > 1 + DK/2 + (K-1)/2$ and $\beta' > e\alpha \beta$, so that  by Lemma \ref{lemma:weights_priori} we have that
\begin{equation}\label{eq:weights_cn}
\Pi\left(\sum_{j > \ceil{\beta' \log n}}w_j > \frac{1}{n^\beta} \right) \leq n^{-\alpha\beta}.
\end{equation}
Reasoning as in the proof of Theorem \ref{thm:tail_stick_post}, it is sufficient to prove that
\begin{equation}\label{eq:suff_weights_cn}
n^{DK/2 + (K-1)/2}\Phi\biggl( \exists \, i = 1, \dots, n \text{ s.t. }c_i >  \lceil \beta' \log n \rceil\biggr) \to 0,
\end{equation}
as $n \to \infty$. By simple calculations and \eqref{eq:weights_cn} we have that
\[
\begin{aligned}
\Phi\biggl( \exists \, i = 1, \dots, n \text{ s.t. }c_i >  \lceil \beta' \log n \rceil\biggr) &= 1 - \E\left[\left(1-\sum_{j > \ceil{\beta' \log n}}w_j \right)^n \right]\\
&\leq 1-\left(1-\frac{1}{n^\beta} \right)^n\left( 1- \frac{1}{n^{\alpha\beta}}\right),
\end{aligned}
\]
which is asymptotically smaller than $2n^{1-\beta}$. Then \eqref{eq:suff_weights_cn} follows by definition of $\beta$.
\end{proof}
\begin{lemma}\label{lm:preliminary_cn_truncated}
Let $\beta' > 0$ as in Lemma \ref{lm:preliminary_cn}. Under assumptions $(A1)-(A3)$ and $(B1)-(B3)$, for every $N$ it holds that
\[
\tilde{\Phi}_N\left(\tilde{c}_i \leq  \lceil \beta' \log n \rceil \text{ for every } i = 1, \dots, n \mid X_{1:n} \right) \to 1,
\]
as $n \to \infty$ in $Q^{(\infty)}$-probability.
\end{lemma}
\begin{proof}
We consider only the case where $N > \lceil \beta' \log n \rceil$, otherwise the result is trivial.

Reasoning exactly as in Lemma \ref{lemma:weights_priori} we have that
\[
\tilde{\Pi}_N\left(\sum_{j > \ceil{\beta' \log n}}\tilde{w}_j > \frac{1}{n^\beta} \right) \leq n^{-\alpha\beta}.
\]
Moreover, reasoning as in Theorem \ref{thm:evidence_lower_bound}, we can deduce that for every $\epsilon > 0$ there exists $c := c(\epsilon, Q) > 0$ such that
\[
Q^{(n)}\left(\int \prod_{i = 1}^n\frac{\tilde{P} f(X_i)}{P^* f(X_i)}\tilde{\Pi}_N(\d \tilde{P}) \geq cn^{-DK/2 - (K-1)/2 -\alpha/2} \right) \geq 1-\epsilon,
\]
for every $n$. Then the proof follows the same steps of Lemma \ref{lm:preliminary_cn}.
\end{proof}
In the following, denote with $\P$ the probability distribution of $(P, c_{1:n})$ on $\mathcal{P}(\Theta) \times \N^n$ induced by model \eqref{DPM_model2}. Similarly, denote with $\tilde{\P}_N$ the probability distribution of $(\tilde{P}, \tilde{c}_{1:n})$. For ease of reference we also denote
\begin{equation}\label{eq:def_B}
B = \left\{c_{1:n} \mid c_i \leq  \lceil \beta' \log n \rceil \,\, \forall i = 1, \dots, n \right\}, \quad \tilde{B} = \left\{\tilde{c}_{1:n} \mid \tilde{c}_i \leq  \lceil \beta' \log n \rceil \,\, \forall i = 1, \dots, n \right\}.
\end{equation}
The next corollary shows that conditioning $P$ and $\tilde{P}$ on $B$ and $\tilde{B}$, respectively, leads to the same probability measure.
\begin{corollary}\label{crl:upper_bound_truncation}
Let $\beta' > 0$ as in Lemma \ref{lm:preliminary_cn}. Under assumptions $(A1)-(A3)$ and $(B1)-(B3)$, for every $N \geq \lceil \beta' \log n \rceil$ we have that
\[
\int_{\Delta_N \times \Theta^N} g(\d w_{1:N}, \d \theta_{1:N}) \Pi \left(\d w_{1:N}, \d \theta_{1:N} \mid B,X_{1:n} \right)= \int_{\Delta_N \times \Theta^N} g(\d \tilde{w}_{1:N}, \d \tilde{\theta}_{1:N})\tilde{\Pi}_N \left(\d \tilde{w}_{1:N}, \d \tilde{\theta}_{1:N} \mid \tilde{B}, X_{1:n}  \right) ,
\]
for every measurable function $g \,:\, \Delta_N \times \Theta^N \, \to \, \R$, where $\Delta_N$ is the $N$-dimensional simplex.
\end{corollary}
\begin{proof}
It suffices to prove the result for $g(\cdot) = \mathbbm{1}_A(\cdot)$, for $A \subset \Delta_N \times \Theta^N$ measurable.

Notice that by definition $\Pi(\d w_{1:N}, \d \theta_{1:N}) = \tilde{\Pi}_N(\d w_{1:N}, \d \theta_{1:N})$ and
\[
\P(c_i = k \mid B, w_{1:N}) = \frac{w_k}{\sum_{j = 1}^{\lceil \beta' \log n \rceil} w_j}, \quad \tilde{\P}_N(\tilde{c}_i = k \mid \tilde{B}, \tilde{w}_{1:N}) = \frac{\tilde{w}_k}{\sum_{j = 1}^{\lceil \beta' \log n \rceil} \tilde{w}_j}.
\]
Therefore
\[
\begin{aligned}
\P \left(A\mid X_{1:n}, B \right) &= \frac{\int_A \sum_{k_{1:n} \in B}\left(\prod_{i = 1}^nf_{\theta_{k_i}}(X_i)\frac{w_{k_i}}{\sum_{j = 1}^{\lceil \beta' \log n \rceil} w_j}\right) \, \Pi(\d w_{1:N}, \d \theta_{1:N})}{\int \sum_{k_{1:n} \in B}\left(\prod_{i = 1}^nf_{\theta_{k_i}}(X_i)\frac{w_{k_i}}{\sum_{j = 1}^{\lceil \beta' \log n \rceil} w_j}\right) \, \Pi(\d w_{1:N}, \d \theta_{1:N})}\\
&= \frac{\int_A \sum_{k_{1:n} \in \tilde{B}}\left(\prod_{i = 1}^nf_{\theta_{k_i}}(X_i)\frac{\tilde{w}_{k_i}}{\sum_{j = 1}^{\lceil \beta' \log n \rceil} \tilde{w}_j}\right) \, \tilde{\Pi}_N(\d \tilde{w}_{1:N}, \d \tilde{\theta}_{1:N})}{\int \sum_{k_{1:n} \in \tilde{B}}\left(\prod_{i = 1}^nf_{\theta_{k_i}}(X_i)\frac{\tilde{w}_{k_i}}{\sum_{j = 1}^{\lceil \beta' \log n \rceil} \tilde{w}_j}\right) \, \tilde{\Pi}_N(\d \tilde{w}_{1:N}, \d \tilde{\theta}_{1:N})} = \tilde{\P}_N \left(A\mid X_{1:n}, \tilde{B} \right),
\end{aligned}
\]
as desired.
\end{proof}
\begin{proof}[Proof of Theorem \ref{thm:conv_clustering}]
As regards the first point, by Lemmas \ref{lm:preliminary_cn} and \ref{lm:preliminary_cn_truncated} it suffices to prove that
\[
 \Phi(k_{1:n} \mid B, X_{1:n}) = \tilde{\Phi}_N(k_{1:n} \mid \tilde{B}, X_{1:n}),
\]
for every $k_{1:n} \in \N^n$, with $B$ and $\tilde{B}$ as in \eqref{eq:def_B}. Notice that by definition
\[
\P(c_i = k \mid B, w_{1:N}, \theta_{1:N}, X_i) = \frac{w_k f_{\theta_k}(X_i)}{\sum_{j = 1}^{\lceil \beta' \log n \rceil} w_jf_{\theta_j}(X_i)}, \quad \tilde{\P}_N(\tilde{c}_i = k \mid \tilde{B}, \tilde{w}_{1:N}, \tilde{\theta}_{1:N}, X_i) = \frac{\tilde{w}_kf_{\tilde{\theta}_k}(X_i)}{\sum_{j = 1}^{\lceil \beta' \log n \rceil} \tilde{w}_jf_{\tilde{\theta}_j}(X_i)},
\]
which therefore means
\[
\begin{aligned}
\Phi(k_{1:n} \mid B, X_{1:n}) = \int_{\Delta_N\times \Theta^N} \left(\prod_{i = 1}^n\frac{w_{k_i} f_{\theta_{k_i}}(X_i)}{\sum_{j = 1}^{\lceil \beta' \log n \rceil} w_jf_{\theta_j}(X_i)}\right)\, \Pi(\d w_{1:N}, \d \theta_{1:N} \mid B, X_{1:n})
\end{aligned}
\]
and
\[
\begin{aligned}
\tilde{\Phi}_N(k_{1:n} \mid \tilde{B}, X_{1:n}) = \int_{\Delta_N\times \Theta^N} \left(\prod_{i = 1}^n\frac{w_{k_i} f_{\theta_{k_i}}(X_i)}{\sum_{j = 1}^{\lceil \beta' \log n \rceil} w_jf_{\theta_j}(X_i)}\right)\, \tilde{\Pi}_N(\d w_{1:N}, \d \theta_{1:N} \mid \tilde{B}, X_{1:n}).
\end{aligned}
\]
The result then follows by Corollary \ref{crl:upper_bound_truncation}.

As regards the second part of the statement, assume by contradiction that there exist $\epsilon > 0$ and $\delta > 0$ such that
\[
Q^{(n)}\left(\left \lvert \left \lvert \Phi \left(c_{1:n} \mid X_{1:n} \right)-\tilde{\Phi}_N \left(\tilde{c}_{1:n} \mid X_{1:n}  \right) \right\rvert \right\rvert_{TV} < 1-\epsilon\right) > \delta,
\]
for every $n$ (or along a suitable diverging subsequence). Thus, with probability at least $\delta$ there exists a coupling $\gamma_n$ of $\Phi \left(c_{1:n} \mid X_{1:n} \right)$ and  $\tilde{\Phi}_N \left(\tilde{c}_{1:n} \mid X_{1:n}  \right)$ such that
\[
\gamma_n\left( c_{1:n} = \tilde{c}_{1:n}\right) > \epsilon.
\]
Assume now that $c_{1:n} = \tilde{c}_{1:n}$, which implies that $c_i \leq N$ for $i = 1, \dots,n$ and thus $K_n \leq N$. Combining all of the above we have that
\[
Q^{(n)}\biggl( \Phi\left(K_n \leq \ceil{\beta \log n} \mid X_{1:n}\right) > \epsilon \biggr) > \delta
\]
for some $\delta > 0$ and every $n$ large enough. Fix $\alpha^*$ and $\underline{\delta}$ as in Theorem \ref{thm:number_clusters_post}. Then, if $\alpha > \alpha^*$, choosing $\beta < (1-\underline{\delta})\alpha$ directly contradicts Theorem \ref{thm:number_clusters_post}.
\end{proof}
\end{appendix}

\renewcommand{\bibname}{References}

\end{document}